\documentclass[aop,MSNbibl,seceqn,dvips]{arximspdf}
\usepackage{mathrsfs}

\doi{10.1214/12-AOP783} 
\volume{41}
\issue{6}
\pubyear{2013}
\firstpage{4317}
\lastpage{4341}

\makeatletter
\newcommand{\eqref}[1]{(\ref{#1})}
\newcommand{\mathset}[1]{\mathbb{#1}}
\newcommand{\F}{\mathcal{F}}
\newcommand{\ca}{c\`adl\`ag}
\newcommand{\R}{\mathset R}
\newcommand{\abs}[2][]{#1\vert #2#1\vert}
\newcommand{\norm}[2][]{#1\Vert #2#1\Vert}
\newcommand{\inner}[2][]{#1\langle #2 #1\rangle}
\newcommand{\N}{\mathset N}
\newcommand{\de}{D([0,1];E)}

\def\E{\mathbb{E}}
\def\P{\mathbb{P}}
\def\1{\mathbf{1}}

\def\eid{\stackrel{d}{=}} 
\def\cid{\stackrel{d}{\rightarrow}} 
\def\cw{\stackrel{w}{\rightarrow}} 

\newtheorem{lemma}{Lemma}[section]

\newtheorem{theorem}[lemma]{Theorem}
\newtheorem{corollary}[lemma]{Corollary}
\newproclaim{example}[lemma]{Example}
\newproclaim{condition}[lemma]{Condition}
\newproclaim{definition}[lemma]{Definition}
\newproclaim{remark}[lemma]{Remark}
\makeatother

\begin{document}
\begin{frontmatter}

\title{On the uniform convergence of random series in Skorohod space
and representations of c\`adl\`ag infinitely divisible processes}
\runtitle{Random series in Skorohod space}

\begin{aug}
\author[A]{\fnms{Andreas} \snm{Basse-O'Connor}\ead[label=e1]{basse@imf.au.dk}}
\and
\author[B]{\fnms{Jan} \snm{Rosi\'nski}\corref{}\ead[label=e2]{rosinski@math.utk.edu}}
\runauthor{A. Basse-O'Connor and J. Rosi\'nski}
\affiliation{Aarhus University and University of Tennessee, and\break
University of Tennessee}
\address[A]{Department of Mathematics\\
Aarhus University\\
8000 Aarhus C\\
Denmark\\
and\\
Department of Mathematics\\
University of Tennessee\\
Knoxville, Tennessee 37996\\
USA\\
\printead{e1}}

\address[B]{Department of Mathematics\\
University of Tennessee\\
Knoxville, Tennessee 37996\\
USA\\
\printead{e2}}
\end{aug}

\received{\smonth{11} \syear{2011}}
\revised{\smonth{5} \syear{2012}}

%
\begin{abstract}
Let $X_n$ be independent random elements in the Skorohod space
$D([0,1]; E)$ of c\`adl\`ag functions taking values in a separable
Banach space $E$. Let $S_n = \sum_{j=1}^{n} X_j$. We show that if $S_n$
converges in finite dimensional distributions to a c\`adl\`ag process,
then $S_n + y_n$ converges a.s. pathwise uniformly over $[0,1]$, for
some $y_n \in D([0,1]; E)$. This result extends the It\^o--Nisio
theorem to the space $D([0,1]; E)$, which is surprisingly lacking in
the literature even for
$E=R$. The main difficulties of dealing with $D([0,1]; E)$ in this
context are its nonseparability under the uniform norm and the
discontinuity of addition under Skorohod's $J_1$-topology.

We use this result to prove the uniform convergence of various series
representations of c\`adl\`ag infinitely divisible processes. As a
consequence, we obtain explicit representations of the jump process,
and of related path functionals, in a general non-Markovian setting.
Finally, we illustrate our results on an example of stable processes.
To this aim we obtain new criteria for such processes to have c\`adl\`
ag modifications, which may also be of independent interest.
\end{abstract}

%

%

\end{frontmatter}

\section{Introduction}

The It\^o--Nisio theorem \cite{Ito-Nisio-theorem} plays a fundamental
role in the study of
series of independent random vectors in separable Banach spaces; see,
for example,
Araujo and Gin{\'e} \cite{A-G-book}, Linde \cite{Linde}, Kwapie{\'n}
and Woyczy{\'n}ski \cite{KwapienRS} and Ledoux and Talagrand \cite
{Talagrand}. In particular, it implies that various series expansions
of a Brownian motion, and of other sample continuous Gaussian
processes, converge uniformly pathwise, which was the original
motivation for the theorem; see Ikeda and Taniguchi \cite{Ikeda-T}.

In order to obtain the corresponding results for series expansions of
sample discontinuous processes, it is natural to consider an extension
of the It\^o--Nisio theorem to the Skorohod space $D[0,1]$ of \ca\
functions. A deep, pioneering work in this direction was done by
Kallenberg \cite{Kallenbergthm}. Among other results, he showed that
if a series
of independent random elements in $D[0,1]$ converges in distribution in
the Skorohod topology, then it ``usually'' converges a.s. uniformly on
$[0,1]$; see Section \ref{I-N-D} for more details.
See also related work \cite{BR88}.
Notice that $D[0,1]$ under the uniform norm $\|\cdot\|$ is not
separable, and such basic random elements in $D[0,1]$ as a Poisson
process are not strongly measurable functions. Therefore, we may
formulate our problem concerning
$(D[0,1], \|\cdot\|)$ in a more general framework of nonseparable
Banach spaces as follows.

Consider a Banach space $(F,\|\cdot\|)$ of functions from a set $T$
into $\R$ such that all evaluation functionals
${\delta_t}\dvtx {x}\mapsto{x(t)}$
are continuous. Assume, moreover, that the map $x\mapsto
\Vert x\Vert $ is measurable with respect to the cylindrical $\sigma
$-algebra $\mathcal{C}(F)=\sigma(\delta_t\dvtx t\in T)$ of $F$. Let $\{
X_j\}
$ be a sequence of independent and \emph{symmetric} stochastic processes
indexed by $T$ with paths in $F$ and set $S_n=\sum_{j=1}^n X_j$. That
is, $S_n$ are $\mathcal{C}(F)$-measurable random vectors in $F$. We
will say that the It\^o--Nisio theorem holds for $F$ if the following
two conditions are equivalent:
\begin{longlist}[(ii)]
\item[(i)]
$S_n$ converges in finite dimensional distributions
to a process with paths in~$F$;
\item[(ii)]
$S_n$ converges a.s. in $(F, \|\cdot\|)$
\end{longlist}
for all sequences $\{X_j\}$ as above.

If $F$ is separable, the It\^o--Nisio theorem gives the equivalence of
(i) and (ii), and in this case $\mathcal{C}(F)=
\mathcal{B}(F)$.
For nonseparable Banach spaces we have examples, but not a general
characterization of spaces for which the It\^o--Nisio theorem holds,
despite the fact that many interesting path spaces occurring in
probability theory are nonseparable.
For instance, the It\^o--Nisio theorem holds for $\mathrm{BV}_1$, the space of
right-continuous functions of bounded
variation, which can be deduced from the proof of Jain and Monrad \cite
{GauQua},
Theorem~1.2, by a conditioning argument. However, this theorem fails to
hold for $F=\ell^\infty(\N)$, and it is neither valid for $\mathrm{BV}_p$, the
space of right-continuous functions of bounded $p$-variation with
$p>1$, or for $C^{0,\alpha}([0,1])$, the space of H\"older continuous
functions of order $\alpha\in(0,1]$; see Remark~\ref{remark-l-infty}.
The case of $F=D[0,1]$ under the uniform norm has been open. Notice
that Kallenberg's result \cite{Kallenbergthm} cannot be applied
because the convergence in (i) is much weaker than the
convergence in the Skorohod topology; see also Remark~\ref{convextight}.

In this paper we show that the It\^o--Nisio theorem holds for the space
$\de$ of \ca\ functions from $[0,1]$ into a separable Banach space $E$
under the uniform norm (Theorem~\ref{thm1}). From this theorem we
derive a simple proof of the above mentioned result of Kallenberg
(Corollary~\ref{cor-1} below). Furthermore,\vadjust{\goodbreak} using Theorem~\ref{thm1} we
establish the uniform convergence of shot noise-type expansions of \ca\ Banach
space-valued infinitely divisible processes (Theorem~\ref
{thID1}). In the last part of this paper, we give applications to
stable processes as an example; see Section~\ref{sec-4}. To this aim,
we establish a new sufficient criterion for the existence of \ca\
modifications of general symmetric stable processes (Theorem~\ref
{pro-ca}) and derive explicit expressions and distributions for several
functionals of the corresponding jump processes.

%
%
\subsection*{Definitions and notation}
In the following, $(\Omega,\F,\P)$ is a complete probability space,
$(E,\abs{ \cdot}_E )$ is a separable Banach space and $D([0,1];E)$ is
the space of \ca\ functions from $[0,1]$ into $E$. (C\`adl\`ag means
right-continuous with left-hand limits.) The space $\de$ is equipped
with the cylindrical $\sigma$-algebra, that is, the smallest $\sigma
$-algebra under which all evaluations $x \mapsto x(t)$ are measurable
for $t\in[0,1]$. A random element in $D([0,1];E)$ is a random function
taking values in $D([0,1];E)$ measurable for the cylindrical $\sigma$-algebra.
$\Vert{x}\Vert=\sup_{t\in[0,1]} \abs{x(t)}_E$ denotes the uniform norm of
$x \in\de$ and $\Delta x(t)=x(t)-x(t-)$ is the size of jump of $x$ at
$t$; the mappings $x \mapsto\Vert{x}\Vert$ and $x \mapsto\Delta x(t)$ are
measurable. For more information on $\de$ we refer to Billingsley
\cite
{Billingsley} and Kallenberg \cite{Kallenberg}. Integrals of
$E$-valued functions
are defined in the Bochner sense. By $\cid, \cw, \eid$ and $\mathcal
L(X)$ we denote, respectively, convergence in distribution, convergence
in law, equality in distribution and the law of the random element $X$.

\section{It\^o--Nisio theorem for $\de$}
\label{I-N-D}

Let $\{X_j\}$ be a sequence of independent random elements in $\de$
and let $S_n=\sum_{j=1}^n X_j$. We study the convergence of $S_n$ in
$\de$ with respect to the uniform topology.

Kallenberg \cite{Kallenbergthm} proved that in $D[0,1]$ endowed with the
Skorohod $J_1$-topology~($E=\R$), convergence a.s. and in distribution
of $S_n$ are equivalent. Moreover, if $S_n$ converges in distribution
relative to the Skorohod topology, then it converges uniformly a.s. under mild conditions, such as, for example, when the limit process
does not have a jump of nonrandom size and location. In concrete
situations, however, a verification of the assumption that $S_n$
converges in distribution in the Skorohod topology can perhaps be as
difficult as a direct proof of the uniform convergence. We prove the
uniform convergence of $S_n$ under much weaker conditions.

\begin{theorem}\label{thm1}
Suppose there exist a random element $Y$ in $\de$ and a dense subset
$T$ of $[0,1]$ such that $1 \in T$ and for any $t_1,\ldots, t_k \in T$
%
\begin{equation}
\label{fd} \bigl(S_n(t_1), \ldots, S_n(t_k)
\bigr) \cid\bigl(Y(t_1),\ldots, Y(t_k)\bigr)\qquad \mbox{as } n
\to \infty.
\end{equation}
Then there exists a random element $S$ in $\de$ with the same
distribution as $Y$ such that:
\begin{longlist}[(iii)]
\item[(i)]
$S_n \to S$ a.s. uniformly on $[0,1]$, provided
$X_n$ are symmetric.\vadjust{\goodbreak}
\item[(ii)]
If $X_n$ are not symmetric, then
%
\begin{equation}
\label{eqcen} S_n + y_n \to S \qquad\mbox{a.s. uniformly on
$[0,1]$}
\end{equation}
for some $y_n \in\de$ such that $\lim_{n\to\infty} y_n(t) =0$ for
every $t \in T$.
\item[(iii)]
Moreover, if the family $\{|S(t)|_E\dvtx t\in T\}$ is
uniformly integrable and the functions $t\mapsto\E ( X_n(t)
)$ belong to $\de$, then one can take in \eqref{eqcen} $y_n$
given by
%
\begin{equation}
\label{eqcen1} y_n(t)=\E \bigl( S(t)-S_n(t) \bigr).
\end{equation}
\end{longlist}
\end{theorem}

The next corollary gives an alternative and simpler proof of the above
mentioned result of Kallenberg \cite{Kallenbergthm}. Our proof relies
on Theorem
\ref{thm1}. Recall that the Skorohod $J_1$-topology on $\de$ is
determined by a metric
\[
\label{S-metric} d(x,y) = \inf_{\lambda\in\Lambda} \max \Bigl\{ \sup_{t\in[0,1]}
\bigl|x(t)-y\circ\lambda(t)\bigr|_E, \sup_{t\in[0,1]} \bigl|\lambda(t) - t\bigr|
\Bigr\},
\]
where $\Lambda$ is the class of strictly increasing, continuous
mappings of $[0,1]$ onto itself; see, for example, \cite{Billingsley},
page~124.

\begin{corollary}\label{cor-1}
If $S_n \cid Y$ in the Skorohod $J_1$-topology, and $Y$ does not have
a jump of nonrandom size and location, then $S_n$ converges a.s. uniformly on $[0,1]$.
\end{corollary}

\begin{pf}
Since $S_n \cid Y$, condition \eqref{fd} holds for
\[
T=\bigl\{t\in(0,1)\dvtx \P\bigl(\Delta Y(t) = 0\bigr)=1\bigr\} \cup\{0, 1\} ;
\]
see \cite{Billingsley}, Section~13. By Theorem~\ref{thm1}(ii)
there exist $\{y_n\}\subseteq\de$ and $S\eid Y$ such that $\Vert
{S_n+y_n-S}\Vert\to0$ a.s. Moreover, $\lim_{n\to\infty} y_n(t)=0$ for
every $t \in T$. We want to show that $\Vert{y_n}\Vert\to0$.

Assume to the contrary that $\limsup_{n\to\infty} \Vert{y_n}\Vert>
\varepsilon
>0$. Then there exist a subsequence $N'\subseteq\N$ and a monotone
sequence $\{t_n\}_{n\in N'} \subset[0,1]$ with $t_n\to t$ such that
$\abs{y_n(t_{n})}_E \geq\varepsilon$ for all $n\in N'$.
Assume that $t_n\uparrow t$ (the case $t_n\downarrow t$ follows
similarly). From the uniform convergence we have that
$S_n(t_n)+y_n(t_n)\to S(t-)$ a.s. ($n\to\infty$, $n\in N'$), and since
$S_n+y_n \cid S$ also in $\de$ endowed with the Skorohod topology, the sequence
\[
\label{} W_n:= \bigl(S_n, S_n+y_n,S_n(t_n)+y_n(t_n)
\bigr),\qquad  n\in N',
\]
is tight in $\de^2\times E$ in the product topology. Passing to a
further subsequence, if needed, we may assume that $\{W_n\}_{n\in N'}$
converges in distribution. By the Skorohod Representation theorem (see, e.g.,
\cite{Billingsley}, Theorem~6.7), there exist random elements $\{Z_n\}_{n\in N'}$
and $Z$ in $\de^2\times E$ such that $Z_n\eid W_n$ and
$Z_n\to Z$ a.s.  From the measurability of addition and the evaluation
maps, it follows that $Z_n$ are on the form
\[
\label{} Z_n= \bigl(U_n,U_n+y_n,U_n(t_n)+y_n(t_n)
\bigr)\vadjust{\goodbreak}
\]
for some random elements $U_n \eid S_n$ in $\de$. We claim that $Z$ is
on the form
%
\begin{equation}
\label{rep-Z} Z= \bigl(U,U,U(t-) \bigr)
\end{equation}
for some random element $U \eid S$ in $\de$. To show this write
$Z=(Z^1,Z^2,Z^3)$ and note that $Z^1\eid Z^2\eid S$. Since the
evaluation map $x \mapsto x(s)$ is continuous at any $x$ such that
$\Delta x(s)=0$ (see
Billingsley \cite{Billingsley}, Theorem 12.5)
for each $s\in T$ with probability one
\[
\label{} Z^1(s)=\lim_{n\to\infty,  n\in N'} U_n(s)=
\lim_{n\to\infty,  n
\in
N'} \bigl[U_n(s)+y_n(s)
\bigr]=Z^2(s),
\]
which shows that $Z^1=Z^2$ a.s. Since $(S_n+y_n,S_n(t_n)+y_n(t_n))\cid
(S,S(t-))$ we {\spaceskip=0.19em plus 0.05em minus 0.02em have that $(S,S(t-))\eid(Z^2,Z^3)$.
The latter yields
$(S(t-),S(t-))\eid\break(Z^2(t-),Z^3)$,} so that $Z^3=Z^2(t-)$ a.s. This
shows \eqref{rep-Z} with $U:=Z^1\eid S$, and with probability one we
have that
\[
\label{} U_n\to U \quad\mbox{and}\quad U_n(t_n)+y_n(t_n)
\to U(t-),\qquad n\to\infty, n\in N'.
\]
We may choose a sequence $\{\lambda_n(\cdot,\omega)\}_{n\in N'}$ in
$\Lambda$ such that as $n\to\infty$,
\[
\label{} \sup_{s\in[0,1]}\bigl|U_n(s)-U\bigl(\lambda_n(s)
\bigr)\bigl|_E+\sup_{s\in
[0,1]} \bigl|\lambda_n(s)-s\bigr|\to0
\qquad\mbox{a.s.}
\]
Therefore,
%
\begin{eqnarray}
\label{con-U} \qquad && \bigl|U\bigl(\lambda_n(t_n)
\bigr)-U(t-)+y_n(t_n)\bigr|_E
\nonumber
\\[-8pt]
\\[-8pt]
\nonumber
 &&\qquad \leq\bigl|U\bigl(
\lambda_n(t_n)\bigr)-U_n(t_n)\bigr|_E+\bigl|U_n(t_n)+y_n(t_n)-U(t-)\bigr|_E
\to0\quad \mbox{a.s.}
\end{eqnarray}
Since $\lambda_n(t_n)\to t$ a.s. as $n\to\infty$, $n\in N'$, the
sequence $\{U(\lambda_n(t_n))\}_{n\in N'}$ is relatively compact in $E$
with at most two cluster points, $U(t)$ or $U(t-)$. By \eqref{con-U},
the cluster points for $\{y_n(t_n)\}_{n\in N'}$ are $-\Delta U(t)$ or
$0$ and since $\abs{y_n(t_n)}_E\geq\varepsilon$, we have that
$y_n(t_n)\to-\Delta U(t)$ a.s., $n\in N'$. This shows that $\Delta
U(t)=c$ for some nonrandom $c\in E\setminus\{0\}$, and since $U\eid
S\eid Y$, we have a contradiction.
\end{pf}

To prove Theorem~\ref{thm1} we need the following lemma:

\begin{lemma}\label{lemma1}
Let $\{x_i\} \subseteq\de$ be a deterministic sequence, and let $\{
\varepsilon_i\}$ be i.i.d. symmetric Bernoulli variables. Assume that
there is a dense set $T \subseteq[0,1]$ with $1\in T$ and a random element $S$ in
$\de$ such that for each $t \in T$,
%
\begin{equation}
\label{eq1} S(t) = \sum_{i=1}^{\infty}
\varepsilon_i x_i(t) \qquad\mbox{a.s.}
\end{equation}
Then
%
\begin{equation}
\label{eq2} \lim_{i \to\infty} \|x_i\| = 0.
\end{equation}
\end{lemma}

\begin{pf}
Suppose to the contrary, there is an $\varepsilon>0$ such that
%
\begin{equation}
\label{eq3} \limsup_{i \to\infty} \|x_i\| > \varepsilon.
\end{equation}
Choose $i_1 \in\N$ and $t_1\in T$ such that $\abs{x_{i_1}(t_1)}_E>
\varepsilon$ and then inductively choose
$i_n \in\N$ and $t_n\in T$, $n \ge2$, such that
\[
\label{eq4} \bigl|x_{i_n}(t_n)\bigr|_E > \varepsilon
\quad\mbox{and} \quad\bigl|x_{i_n}(t_k)\bigr|_E < \varepsilon/2
\qquad\mbox{for all } k < n.
\]
This is always possible in view of \eqref{eq1} and \eqref{eq3} because
$\lim_{i \to\infty} x_i(t)=0$ for each $t \in T$. It follows that all
$t_n$'s are distinct. The sequence $\{t_n\}_{n \in\N}$ contains a
monotone convergent subsequence $\{t_{n}\}_{n \in N'}$, $ \lim_{n \to
\infty,  n \in N'} t_{n} = t$. Then for every $n>k$, $k, n \in N'$,
%
\begin{eqnarray}
\label{L} \P \bigl(\bigl|S(t_{n}) -S(t_{k})\bigr|_E >
\varepsilon/2 \bigr) &=& \P \Biggl(\Biggl\vert\sum_{i=1}^{\infty}
\varepsilon_i\bigl[x_i(t_{n})
-x_i(t_{k})\bigr]\Biggr\vert_E > \varepsilon/2
\Biggr)
\nonumber
\\[-8pt]
\\[-8pt]
\nonumber
& \ge&\frac{1}{2} \P \bigl(\bigl\vert\varepsilon_{i_n}
\bigl[x_{i_n}(t_{n}) -x_{i_n}(t_{k})\bigr]
\bigr\vert_E > \varepsilon/2 \bigr) = \frac{1}{2} ,
\end{eqnarray}
which follows from the fact that if $(X,Y) \eid(X,-Y)$, then for all
$\tau>0$, $\P(\norm{X} >\tau)= \P(\norm{(X+Y)+(X-Y)} >2\tau) \le
2 \P
(\norm{X+Y} >\tau)$.
Bound \eqref{L} contradicts the fact that $S$ is \ca\  and thus proves
\eqref{eq2}.
\end{pf}

\begin{pf*}{Proof of Theorem~\ref{thm1}} First we construct a random
element $S$ in $\de$ such that $S \eid Y$ and
%
\begin{equation}
\label{eqS} S(t)=\lim_{n\to\infty} S_n(t) \qquad\mbox{a.s. for every } t
\in T.
\end{equation}
By the It\^o--Nisio theorem~\cite{Ito-Nisio-theorem}, $ S^*(t)=\lim_{n\to\infty} S_n(t)$
exists a.s. for $t \in T$.
Put $ S^*(t)=\lim_{r\downarrow t,  r \in T} S^*(r)$ when $t \in
[0,1]\setminus T$, where the limit is in probability [the limit exists
since $ (S^*(r),S^*(s)) \eid(Y(r),Y(s))$ for all $r,s \in T$ and $Y$
is right-continuous].
Therefore, the process $ \{S^*(t)\}_{t\in[0,1]}$ has the same finite
dimensional distributions as $ \{Y(t)\}_{t\in[0,1]}$ whose paths are
in $\de$. Since the cylindrical $\sigma$-algebra of $ \de$  coincides
with the Borel $\sigma$-algebra under the Skorohod topology, by
Kallenberg \cite{Kallenberg}, Lemma 3.24,  there is a process $ S=\{
S(t)\}_{t\in
[0,1]}$, on the same probability space as $S^*$, with all paths in $\de
$ and such that $ \P(S(t)=S^*(t))=1$ for every $t\in[0,1]$.

(i): Let $n_1<n_2<\cdots$ be an arbitrary subsequence in $\N
$ and
$\{\varepsilon_i\}$ be i.i.d.  symmetric Bernoulli variables defined on
$(\Omega',\F',\P')$. By the symmetry, $W_k$ in $\de$ given by
\[
W_k(t) = \sum_{i=1}^{k}
\varepsilon_i \bigl(S_{n_{i}}(t) - S_{n_{i-1}}(t)\bigr), \qquad t
\in[0,1],
\]
($S_{n_0}\equiv0$) has the same distribution as $S_{n_k}$. By the
argument stated at the beginning of the proof, there is a process $W=\{
W(t)\}_{t\in[0,1]}$ with paths in $\de$, defined on $(\Omega'\times
\Omega,\F'\otimes\F, \P'\otimes\P)$, such that $W \eid Y$ and
\[
\label{eqW} W(t) = \sum_{i=1}^{\infty}
\varepsilon_i \bigl(S_{n_{i}}(t) - S_{n_{i-1}}(t)\bigr)
\qquad\mbox{a.s. for every } t\in T.
\]

Choose a countable set $T_0\subset T$, dense in $[0,1]$ with $1\in T_0$, and $\Omega_0
\subseteq\Omega$, $\P(\Omega_0)=1$, such that for each $\omega\in
\Omega_0$,  $\P'\{\omega'\dvtx W(\cdot, \omega', \omega) \in\de\}
=1$ and
\[
\label{} W(t, \cdot, \omega) = \sum_{i=1}^{\infty}
\varepsilon_i \bigl(S_{n_{i}}(t,\omega ) - S_{n_{i-1}}(t,
\omega)\bigr)\qquad \P'\mbox{-a.s. } \mbox{for every $t\in
T_0$}.
\]
By Lemma~\ref{lemma1}, $\lim_{i\to\infty} \|S_{n_{i}}(\omega) -
S_{n_{i-1}}(\omega)\| = 0$, which implies that $\|S_n -S\| \to0$ in
probability. By the L\'evy--Octaviani inequality
\cite{KwapienRS}, Proposition 1.1.1(i), which holds for measurable
seminorms on linear measurable spaces, $\|S_n - S\| \to0$ almost surely.

(ii):
Define on the product probability space $(\Omega\times\Omega, \F
\otimes\F, \P\otimes\P)$ the following: $ \tilde{X}_n(t; \omega,
\omega')= X_n(t, \omega)-X_n(t, \omega')$, $ \tilde{S}(t; \omega,
\omega')=S(t, \omega)-S(t, \omega')$ and $ \tilde{S}_n = \sum_{k=1}^n
\tilde
{X}_k$, where the random element $S$ in $\de$ is determined by~\eqref
{eqS}. By (i), $ \tilde{S}_n \to\tilde{S}$ a.s. in $\|
\cdot\|
$. From Fubini's theorem we infer that there is an $\omega'$ such that
the functions $x_n(\cdot)= X_n(\cdot, \omega')$ and $y(\cdot)=
S(\cdot
,\omega')$ belong to $\de$ and $\sum_{k=1}^n (X_k - x_k) \to S - y$
a.s. in $\|\cdot\|$. Thus \eqref{eqcen} holds with $ y_n=y-\sum_{k=1}^n x_k$, which combined with \eqref{eqS} yields $\lim_{n \to
\infty} y_n(t)=0$ for every $t \in T$.

(iii):
Let us assume for a moment that $\E S(t)= \E S_n(t) =0$ for all $t\in
T$ and $n \in\N$. We want to show that $y_n=0$ satisfies \eqref
{eqcen}. Since $S(t)\in L^1(E)$ we have that $S_n(t) \to S(t)$ in
$L^1(E)$ (cf. \cite{KwapienRS}, Theorem~2.3.2) and hence $S_n(t)=\E
[S(t) | \mathcal{F}_n]$ where $\mathcal{F}_n= \sigma(X_1,\ldots,X_n)$.
This shows that $\{S_n(t) \dvtx t\in T, n\in\N\}$ is uniformly integrable;
cf. \cite{JHJBook}, (6.10.1).
First we will prove that the sequence $\{y_n\}$ is uniformly bounded,
that is,
%
\begin{equation}
\label{eqfn-bdd} \sup_{n\in\N} \|y_n\|< \infty.
\end{equation}
Assume to the contrary that there exists an increasing subsequence $n_i
\in\N$ and $t_i\in T$ such that
%
\begin{equation}
\label{eqfn-unb} \bigl|y_{n_i}(t_i)\bigr|_E >
i^3,\qquad i \in\N.
\end{equation}
Define
\[
\label{} \mathbf{V}_n = \bigl(S_n(t_1),
\ldots, i^{-2}S_n(t_i), \ldots\bigr).
\]
$\mathbf{V}_n$ are random vectors in $c_0(E)$ since
\[
\label{} \E\limsup_{k\to\infty} \bigl|k^{-2}S_n(t_k)\bigr|_E
\le\lim_{k\to\infty} \sum_{i=k}^\infty
i^{-2} \E\bigl|S_n(t_i)\bigr|_E \le
\lim_{k\to\infty} M \sum_{i=k}^\infty
i^{-2} =0,
\]
where $M=\sup_{t\in T} \E|S(t)|_E$. By the same argument,
\[
\label{} \mathbf{V} = \bigl(S(t_1), \ldots, i^{-2}S(t_i),
\ldots\bigr)
\]
is a random vector in $c_0(E)$, and since $S_n(t_i) \to S(t_i)$ in
$L^1(E)$,  $\E\|\mathbf{V}_n-\mathbf{V}\|_{c_0(E)} \to0$.
Thus $\mathbf{V}_n \to\mathbf{V}$ a.s. in $c_0(E)$ by It{\^o} and
Nisio \cite{Ito-Nisio-theorem}, Theorem~3.1.

Since each $y_n$ is a bounded function,
\[
\label{} \mathbf{a}_n = \bigl(y_n(t_1),
\ldots, i^{-2}y_n(t_i), \ldots\bigr) \in
c_0(E).
\]
Also $\mathbf{V}_n + \mathbf{a}_n \to\mathbf{V}$ a.s. in $c_0(E)$ because
\[
\label{} \|\mathbf{V}_n + \mathbf{a}_n - \mathbf{V}
\|_{c_0(E)} \le\| S_n+y_n-S\| \to0.
\]
Hence $\mathbf{a}_n=(\mathbf{V}_n + \mathbf{a}_n)- \mathbf{V}_n \to0$
in $c_0(E)$. Since $\lim_{i\to\infty} \|\mathbf{a}_{n_i}\|_{c_0(E)} =
\infty$ by~\eqref{eqfn-unb}, we have a contradiction. Thus \eqref
{eqfn-bdd} holds.

Now we will show that
%
\begin{equation}
\label{eqfn-conv} \lim_{n\to\infty} \|y_n\| = 0.
\end{equation}
Assume to the contrary that there exists an $\varepsilon>0$, an increasing
subsequence \mbox{$n_i \in\N$}, and $t_i\in T$ such that
%
\begin{equation}
\label{eqfn-unc} \bigl|y_{n_i}(t_i)\bigr|_E >
\varepsilon,\qquad i \in\N.
\end{equation}
Since \eqref{eqfn-bdd} holds, $\{S_n(t)+y_n(t) \dvtx t\in T, n\in\N\}$ is
uniformly integrable. Passing to a subsequence, if necessary, we may
assume that $\{t_i\}$ is strictly monotone and converges to some $t\in[0,1]$.
It follows from \eqref{eqcen} that $S_{n_i}(t_i) + y_{n_i}(t_i) \to Z$
a.s. in $E$, where $Z=S(t)$ or $Z=S(t{-})$. By the uniform
integrability the convergence also holds in $L^1(E)$, thus
$y_{n_i}(t_i) \to\E Z=0$, which contradicts \eqref{eqfn-unc}.

We proved \eqref{eqfn-conv}, so that \eqref{eqcen} holds with $y_n=0$
when $\E S(t) = \E S_n(t)=0$ for all $t\in T$ and $n \in\N$.
In the general case, notice that $\E S(\cdot) \in\de$, so that $S-\E S
\in\de$. From the already proved mean-zero case,
\[
\label{} \sum_{k=1}^n (X_k -
\E X_k) \to S - \E S \qquad\mbox{a.s. uniformly on [0,1], }
\]
which gives \eqref{eqcen} and \eqref{eqcen1}.
\end{pf*}

Next we will show that the It\^o--Nisio theorem does not hold in many
interesting nonseparable Banach spaces. From this perspective, the
spaces $\mathrm{BV}_1$ and $(\de, \|\cdot\|)$ are exceptional. We will use the
following notation.\vadjust{\goodbreak}

For $p\geq1$, $\mathrm{BV}_p$ is the space of right-continuous functions
${f}\dvtx[0,1]\to{\R}$
of bounded $p$-variation with $f(0)=0$ equipped with
the norm
\[
\label{} \norm{f}_{\mathrm{BV}_p}=\sup \Biggl\{ \Biggl(\sum
_{j=1}^n \bigl|f(t_j)-f(t_{j-1})\bigr|^p
\Biggr)^{1/p}\dvtx n\in\mathset N, 0= t_0\leq\cdots \leq
t_n=1 \Biggr\}.
\]
For $\alpha\in(0,1]$, $C^{0,\alpha}([0,1])$ is the space of $\alpha
$-H\"older continuous functions
${f}\dvtx [0,1]\to{\R}$ with $f(0)=0$
equipped with the norm
\[
\label{} \norm{f}_{C^{0,\alpha}}=\sup_{s,t\in[0,1]\dvtx s\neq t} \frac{\abs
{f(t)-f(s)}}{\abs{t-s}^\alpha}.
\]
Moreover, $\ell^\infty(\N)$ is the space of real sequences $\mathbf
a=\{
a_k\}_{k\in\N}$ with the norm $\norm{\mathbf a}_{\ell^\infty
}:=\sup_{k\in\N} \abs{a_k}<\infty$.

\begin{remark}\label{remark-l-infty}
In the following we will show that the It\^o--Nisio theorem is not
valid for the following nonseparable Banach spaces: $\ell^\infty(\N)$,
$\mathrm{BV}_p$ for $p>1$ and $C^{0,\alpha}([0,1])$ for $\alpha\in(0,1]$.

For all $p>1$ set $r=4^{[p/(p-1)+1]}$ where $[ \cdot]$ denotes the
integer part. For $j\in\N$ let
\[
\label{} f_j(t)=r^{-j/p}\log^{-1/2}(j+1)\sin
\bigl(r^j \pi t\bigr),\qquad t\in[0,1],
\]
$\{Z_j\}$ be i.i.d. standard Gaussian random variables, and $X=\{X(t)\}_{t\in[0,1]}$ be given by
%
\begin{equation}
\label{se-X} X(t)=\sum_{j=1}^\infty
f_j(t) Z_j \qquad\mbox{a.s. }
\end{equation}
According to Jain and Monrad \cite{JainMonradBp}, Proposition~4.5, $X$
has paths in
$\mathrm{BV}_p$, but series~\eqref{se-X} does not converge in $\mathrm{BV}_p$. This shows
that the It\^o--Nisio theorem is not valid for $\mathrm{BV}_p$ for $p>1$. A
closer inspection of
\cite{JainMonradBp}, Proposition~4.5, reveals that~$X$, given by
\eqref
{se-X}, has paths in $C^{0,1/p}([0,1])$ and since $\|\cdot\|_{\mathrm{BV}_p}\leq
\|\cdot\|_{C^{0,1/p}}$, the It\^o--Nisio theorem is not valid for
$C^{0,\alpha}([0,1])$ with $\alpha\in(0,1)$.

For fixed $p>1$ choose a sequence $\{x_n^*\}_{n\in\N}$ of continuous
linear mappings from $\mathrm{BV}_p$ into $\R$, each of the form
\[
\label{} x\mapsto\sum_{i=1}^k
\alpha_i \bigl(x(t_i)-x(t_{i-1}) \bigr),
\]
where $k\in\N$, $(\alpha_i)_{i=1}^k \subseteq\R$, $\sum_{i=1}^k
\abs
{\alpha_i}^q\leq1$ with $q:=p/(p-1)$ and $0=t_0<\cdots<t_k=1$, such
that
\[
\label{} \norm{f}_{\mathrm{BV}_p}=\sup_{n\in\N}\bigl \vert x_n^*(f)
\bigr\vert \qquad\mbox{for all }f\in \mathrm{BV}_p.
\]
Set $Y(n)=x^*_n(X)$ and $b_j(n)=x^*_n(f_j)$ for all $n,j\in\N$.
Process $Y=\break\{Y(n)\}_{n\in\N}\in\ell^\infty(\N)$ a.s., $b_j=\{
b_j(n)\}_{n\in\N}\in\ell^\infty(\N)$, and since\vadjust{\goodbreak} each $x_n^*$ only depends on
finitely many coordinate variables, we have that
\[
\label{} Y(n)=\sum_{j=1}^\infty
Z_j b_j(n)\qquad \mbox{a.s. for all }n\in\N.
\]
By the identity
\[
\label{} \Biggl\Vert\sum_{j=r}^m
Z_j b_j\Biggr \Vert_{\ell^\infty}= \Biggl\Vert\sum
_{j=r}^m Z_j f_j
\Biggr\Vert_{\mathrm{BV}_p} \qquad\mbox{for }1\leq r<m,
\]
we see that the sequence $\{\sum_{j=1}^n Z_j b_j\}$ is not Cauchy in
$\ell^\infty(\N)$ a.s. and therefore not convergent in $\ell^\infty(\N
)$. This shows that the It\^o--Nisio theorem is not valid for $\ell^\infty(\N)$.

Next we will consider $C^{0,1}([0,1])$.
A function
${f}\dvtx[0,1]\to{\R}$
with $f(0)=0$ is in $C^{0,1}([0,1])$
if and only if it is absolutely continuous with a derivative $f'$ in
$L^\infty([0,1])=L^\infty([0,1],ds)$, and in this case we have
%
\begin{equation}
\label{norm-c1} \norm{f}_{C^{0,1}}=\bigl\Vert f'
\bigr\Vert_{L^\infty}.
\end{equation}
Let $Y=\{Y(n)\}_{n\in\N}$ and $b_j$, for $j\in\N$, be defined as
above and choose a Borel measurable partition $\{A_j\}_{j\in\N}$ of
$[0,1]$ generating $\mathscr B([0,1])$. For all $j,n\in\N$ and $t\in
A_n$ let $h_j(t)=b_j(n)$ and $U(t)=Y(n)$. Then $h_j\in L^\infty([0,1])$
for all $j\in\N$ and $U\in L^\infty([0,1])$ a.s. For all $n\in\N$,
let $y_n^*$ denote the continuous linear functional on $L^1([0,1])$
given by $f\mapsto\int_{A_n} f(s) \,ds$. Since $\{y_n^*\}$ separates
points on $L^1([0,1])$ and
\[
\label{} y_n^*(U)= Y(n)\int_{A_n} 1 \,ds=\sum
_{j=1}^\infty y_n^*(h_j)
Z_j \qquad\mbox{a.s.},
\]
it follows by the It\^o--Nisio theorem
that the series $\sum_{j=1}^\infty h_j Z_j$ converges a.s. in the
separable Banach space $L^1([0,1])$ to $U$,
and hence for all $t\in[0,1]$,
\[
\label{} V(t):=\int_0^t U(s) \,ds=\sum
_{j=1}^\infty Z_j \int_0^t
h_j(s) \,ds \qquad\mbox{a.s.}
\]
Process $V=\{V(t)\}_{t\in[0,1]}\in C^{0,1}([0,1])$ a.s., and for all
$1\leq r\leq v$ we have by~\eqref{norm-c1}
\[
\label{Cauchy-ex} \Biggl\Vert \sum_{j=r}^v
\biggl(\int_0^\cdot h_j(s) \,ds \biggr)
Z_j \Biggr\Vert_{C^{0,1}}=\Biggl\Vert \sum_{j=r}^v
h_j Z_j \Biggr\Vert_{L^\infty}=\Biggl\Vert \sum
_{j=r}^v b_j Z_j
\Biggr\Vert_{\ell^\infty}.
\]
This shows that the It\^o--Nisio theorem is not valid for $C^{0,1}([0,1])$.
\end{remark}

\begin{remark}\label{convextight}
Here we will indicate why the usual arguments in the proof of the It\^
o--Nisio theorem do not work for $D[0,1]$ equipped with Skorohod's
$J_1$-topology. Such\vadjust{\goodbreak} arguments rely on the fact that all probability
measures $\mu$ on a separable Banach space $F$ are convex tight, that
is, for all $\varepsilon>0$ there exists a convex compact set $K\subseteq
F$ such that $\mu(K^c)<\varepsilon$; see, for example, \cite{KwapienRS},
Theorem~2.1.1. This is not the case in $D[0,1]$. We will show that if
$X$ is a continuous in probability process with paths in $D[0,1]$
having convex tight distribution, then $X$ must have continuous sample
paths a.s. Indeed,
let $K$ be a convex compact subset of $D[0,1]$ relative to Skorohod's
$J_1$-topology. According to Daffer and Taylor \cite
{Taylor-LLN-D-space}, Theorem~6, for
every $\varepsilon>0$ there exist $n\in\N$ and $t_1,\ldots,t_n\in[0,1]$
such that for all $x\in K$ and $t\in[0,1]\setminus\{t_1,\ldots,t_n\}$
we have $\abs{\Delta x(t)}\leq\varepsilon$. In particular,
%
\begin{equation}
\label{conv-tight} \P(X\in K)\leq\P \Bigl(\sup_{t\in[0,1]\setminus\{t_1,\ldots,t_n\}
}\bigl\vert\Delta X(t)
\bigr\vert \leq\varepsilon \Bigr)=\P \Bigl(\sup_{t\in[0,1]}\bigl\vert\Delta X(t)\bigr\vert \leq
\varepsilon \Bigr),\hspace*{-35pt}
\end{equation}
where the last equality uses that $X$ is continuous in probability.
Letting $\varepsilon\to0$ on the right-hand side of \eqref{conv-tight}
and taking $K$ such that the left-hand side is close to~1, we prove
that $\P(\sup_{t\in[0,1]} \abs{\Delta X(t)} =0)=1$. Therefore, the
only convex tight random elements in $D[0,1]$, which are continuous in
probability, are sample continuous. In particular, a L\'evy process
with a nontrivial jump part is not convex tight.
\end{remark}

%

\section{Series representations of infinitely divisible processes}
\label{series-rep}

In this section we study infinitely divisible processes with values in
a separable Banach space $E$. Recall that an infinitely divisible
probability measure $\mu$ on $E$, without Gaussian component, admits a
L\'evy--Khintchine representation of the form
%
\begin{eqnarray}
\label{eqLK} \qquad\hat\mu\bigl(x^*\bigr) = \exp \biggl\{i\bigl\langle{x^*,b}\bigr
\rangle + \int_{E} \bigl(e^{i\langle{x^*,
x}\rangle} - 1 - i
\bigl\langle{x^*, [\![ x]\!]}\bigr\rangle \bigr) \nu(dx) \biggr\},
\nonumber
\\[-8pt]
\\[-8pt]
\eqntext{x^* \in
E^*,}
\end{eqnarray}
where $b\in E$, $\nu$ is a $\sigma$-finite measure on $E$ with $\nu
(\{
0\})=0$, and $[\![ x]\!]= x/(1\vee\|x\|)$ is a continuous truncation
function. Vector $b$ will be called the shift and $\nu$ the L\'evy
measure of $\mu$. Here $E^*$ denotes the dual of $E$ and $\inner{x^*,
x} := x^*(x)$, $x^*\in E^*$ and $x\in E$. We refer the reader to \cite
{A-G-book} for more information on infinitely divisible distributions
on Banach spaces.

Let $T$ be an arbitrary set. An $E$-valued stochastic process $X=\{
X(t)\}_{t\in T}$ is called infinitely divisible if for any $t_1,\ldots
,t_n \in T$ the random vector $(X(t_1),\break \ldots,  X(t_n))$ has infinitely
divisible distribution in $E^n$. We can write its characteristic
function in the form
%
\begin{eqnarray}
\label{eqLK1}  && \E\exp \Biggl\{i \sum_{j=1}^n
\bigl\langle{x^*_j,X(t_j)}\bigr\rangle \Biggr\}
\nonumber\hspace*{-15pt}\\
&&\qquad = \exp
\Biggl\{ i \sum_{j=1}^n \bigl
\langle{x^*_j,b(t_j)} \bigr\rangle\hspace*{-15pt}
 \\
 &&\hspace*{45pt}{}+ \int_{E^n} \Biggl( e^{i \sum_{j=1}^n \langle{x^*_j,x_j}\rangle} -1 - i \sum
_{j=1}^n \bigl\langle{x^*_j, [\![
x_j]\!]}\bigr\rangle \Biggr) \nu_{t_1,\ldots,t_n}(dx_1
\cdots dx_n) \Biggr\},\nonumber\hspace*{-15pt}
\end{eqnarray}
where $ \{x_j^*\}\subseteq E^*$, $\{b(t_j)\} \subseteq E$ and $\nu_{t_1,\ldots,t_n}$ are L\'evy measures on $E^n$. Below we will work with
$T=[0,1]$; extensions to $T=[0,a]$ or $T=[0,\infty)$ are obvious.

In this section $\{V_j\}$ will stand for an i.i.d. sequence of random
elements in a measurable space $\mathcal{V}$ with the common
distribution $\eta$.
$\{\Gamma_j\}$ will denote a sequence of partial sums of standard
exponential random variables independent of the sequence $\{V_j\}$. Put $V=V_1$.

\begin{theorem}\label{thID1}
Let $X=\{X(t)\}_{t\in[0,1]}$ be an infinitely divisible process
without Gaussian part specified by \eqref{eqLK1} and with trajectories
in $\de$. Let
${H}\dvtx [0,1]\times\R_{+}\times\mathcal{V} \to{E}$
be a
measurable function such that for every $t_1,\ldots,t_n \in[0,1]$ and
$B \in\mathcal{B}(E^n)$
%
\begin{equation}
\label{eqnu} \int_0^{\infty} \P \bigl(
\bigl(H(t_1, r, V),\ldots,H(t_n, r, V)\bigr) \in B
\setminus\{0\} \bigr) \,dr = \nu_{t_1,\ldots,t_n}(B),
\end{equation}
$H(\cdot, r, v)\in\de$ for every $(r, v) \in\R_+\times\mathcal{V}$,
and $r \mapsto\|H(\cdot, r, v)\|$ is nonincreasing for every $v \in
\mathcal{V}$.
Define for $u>0$,
\[
\label{eqYu} Y^u(t) = b(t) + \sum_{j: \Gamma_j \le u}
H(t, \Gamma_j, V_j) - A^u(t),
\]
where
\[
\label{eqA} A^{u}(t) = \int_0^{u}
\E\bigl[\!\bigl[ H(t,r,V)\bigr]\!\bigr] \,dr .
\]
Then, with probability~1 as $u \to\infty$,
%
\begin{equation}
\label{conv} Y^u(t) \to Y(t)
\end{equation}
uniformly in $t \in[0,1]$, where the process $Y=\{Y(t)\}_{t\in[0,1]}$
has the same finite dimensional distributions as $X$ and paths in $\de$.

Moreover, if the probability space on which the process $X$ is defined
is rich enough, so that there exists a standard uniform random variable
independent of~$X$, then the sequences $\{\Gamma_j, V_j\}$ can be
defined on the same probability space as $X$, such that with
probability~1, $X$ and $Y$ have identical sample paths.
\end{theorem}

The proof of Theorem~\ref{thID1} will be preceded by corollaries,
remarks and a crucial lemma.

%

\begin{corollary}\label{corseries}
Under assumptions and notation of Theorem \ref{thID1}, with probability~1
%
\begin{equation}
\label{eqY1} Y(t) = b(t) + \sum_{j=1}^{\infty}
\bigl[ H(t, \Gamma_j, V_j) - C_j(t) \bigr]\qquad
\mbox{for all } t \in[0,1],
\end{equation}
where the series converges a.s. uniformly on $[0,1]$ and $C_j(t)=
A^{\Gamma_j}(t)-A^{\Gamma_{j-1}}(t)$.

Moreover, if $b$ and $A^{u}$, for sufficiently large $u$, are
continuous functions of $t \in[0,1]$, then with probability~1
%
\begin{equation}
\label{eqjumpY1} \Delta Y(t) = \sum_{j=1}^\infty
\Delta H(t,\Gamma_j, V_j)\qquad \mbox{for all } t \in[0,1],
\end{equation}
where the series converges a.s. uniformly on $[0,1]$. [$\Delta
f(t)=f(t)-f(t-)$ denotes the jump of a function $f \in\de$.]
\end{corollary}
\begin{pf}
Since the convergence in \eqref{conv} holds for a continuous index
$u$, we may take $u=\Gamma_n$, which gives
\[
\label{} Y(t) = \lim_{n \to\infty} Y^{\Gamma_n}(t) = \lim_{n \to\infty}
\Biggl( b(t) + \sum_{j=1}^n H(t,
\Gamma_j, V_j) - A^{\Gamma_n}(t) \Biggr) \qquad\mbox{a.s. in
} \|\cdot\|,
\]
proving \eqref{eqY1}. This argument and our assumptions imply \eqref
{eqjumpY1} as well.
\end{pf}
%

\begin{corollary}\label{corsym}
Suppose that the process $X$ in Theorem \ref{thID1} is symmetric, and
$H$ satisfies stated conditions except that \eqref{eqnu} holds for some
measures $\nu^0_{t_1,\ldots,t_n}$ in place of $\nu_{t_1,\ldots,t_n}$
such that
\[
\label{} \nu_{t_1,\ldots,t_n}(B) = \tfrac{1}{2} \nu^0_{t_1,\ldots,t_n}(B)
+ \tfrac
{1}{2} \nu^0_{t_1,\ldots,t_n}(-B)
\]
for every $B \in\mathcal{B}(E^n)$. Let $\{\varepsilon_j\}$ be i.i.d. symmetric Bernoulli variables independent of $\{\Gamma_j, V_j\}$. Then,
with probability~1, the series
%
\begin{equation}
\label{eqYsym} Y(t) = \sum_{j=1}^{\infty}
\varepsilon_j H(t, \Gamma_j, V_j)
\end{equation}
converges uniformly in $t \in[0,1]$. The process $Y=\{Y(t)\}_{t\in
[0,1]}$ has the same finite dimensional distributions as process $X$
and paths in $\de$.
\end{corollary}
\begin{pf}
Apply Theorem \ref{thID1} for $\tilde{H}\dvtx [0,1]\times\R_+ \times
\tilde{\mathcal{V}} \mapsto E$ defined by
\[
\label{} \tilde{H}(t, r, \tilde{v}) = s H(t, r, v),
\]
where $\tilde{v} = (s,v) \in\tilde{\mathcal{V}}:= \{-1,1\} \times
\tilde{V}$, and $\tilde{V}_j = (\varepsilon_j, V_j)$ in the place of $H$
and~$V_j$.

An alternative way to establish the uniform convergence in \eqref
{eqYsym} is to use Theorem \ref{thm1}(i) conditionally on the sequence
$\{\Gamma_j, V_j\}$.
\end{pf}

\begin{remark}\label{rem1}
There are several ways to find $H$ and $V$ for a given process such
that \eqref{eqnu} is satisfied; see Rosi{\'n}ski \cite
{RosinskiOnSeries} and \cite
{Rosinskiseriespoint}. They lead to different series representations of
infinitely divisible processes. One of such representations will be
given in the next section.
\end{remark}

\begin{lemma}\label{lemnec}
In the setting of Theorem~\ref{thID1}, the assumption that $X$ has
paths in $D([0,1];E)$ implies that $b \in\de$,
%
\begin{equation}
\label{eqint} \int_0^{\infty} \P\bigl(\bigl \|H(\cdot,
r, V)\bigr\|>1 \bigr) \,dr < \infty
\end{equation}
and
%
\begin{equation}
\label{eqG0} \lim_{j \to\infty}\bigl \|H(\cdot, \Gamma_j,
V_j)\bigr\| = 0 \qquad\mbox{a.s.}
\end{equation}
\end{lemma}
\begin{pf}
By the uniqueness, $b=b(\mu)$ in \eqref{eqLK} and by \cite
{RosinskiDissertationes}, Lemma 2.1.1, $\mu_n \cw\mu$ implies $b(\mu_n) \to b(\mu)$ in $E$.
Since $X$ has paths in $D([0,1];E)$, the function $t \mapsto\mathcal
{L}(X(t))$ is \ca, so that $b= b(\mathcal{L}(X(t))) \in\de$.

To prove \eqref{eqint} consider $\tilde{X}(t)= X(t)-X'(t)$, where $X'$
is an independent copy of $X$.
Let $\{\varepsilon_j\}$ be i.i.d.  symmetric Bernoulli variables
independent of $\{(\Gamma_j, V_j)\}$.
Using \cite{RosinskiOnSeries}, Theorem~2.4 and \eqref{eqnu}, we can
easily verify that the series
\[
\label{eqtildeY} \sum_{j=1}^{\infty}
\varepsilon_j H\bigl(t, 2^{-1}\Gamma_j,
V_j\bigr)
\]
converges a.s. for each $t \in[0,1]$ to a process $\tilde{Y}=\{
\tilde
{Y}(t)\}_{t\in[0,1]}$ which has the same finite dimensional
distributions as $\tilde{X}$. Thus we can and do assume that $\tilde
{Y}$ has trajectories in $\de$ a.s. Applying Lemma \ref{lemma1}
conditionally, for a fixed realization of $\{(\Gamma_j, V_j)\}$, we
obtain that
%
\begin{equation}
\label{eqG1} \lim_{j \to\infty} \bigl\|H\bigl(\cdot, 2^{-1}
\Gamma_j, V_j\bigr)\bigr\| = 0 \qquad\mbox{a.s.}
\end{equation}
Observe that for each $\theta\in(2^{-1},1)$,  $\Gamma_j < 2\theta j$
eventually a.s. Thus, by \eqref{eqG1} and the monotonicity of $H$,
\[
\label{} \lim_{j \to\infty}\bigl \|H(\cdot, \theta j, V_j)\bigr\| = 0\qquad
\mbox{a.s.}
\]
By the Borel--Cantelli lemma,
%
\begin{equation}
\label{eqBC} \sum_{j=1}^{\infty} \P\bigl( \bigl\|H(
\cdot, \theta j, V_j)\bigr\| > 1\bigr) < \infty.
\end{equation}
Hence
\begin{eqnarray*}
& &\sum_{j=1}^{\infty} \P\bigl(\bigl \|H(\cdot,
\Gamma_j, V_j)\bigr\| > 1\bigr)
\\
& &\qquad\le\sum_{j=1}^{\infty} \P\bigl( \bigl\|H(\cdot,
\Gamma_j, V_j)\bigr\| > 1, \Gamma_j > \theta j
\bigr) + \sum_{j=1}^{\infty} \P(
\Gamma_j \le\theta j)
\\
&&\qquad \le\sum_{j=1}^{\infty} \P\bigl(\bigl \|H(\cdot,
\theta j, V_j)\bigr\| > 1\bigr) +(1-\theta )^{-1} + \sum
_{j \ge(1-\theta)^{-1}} \frac{(\theta j)^{j}}{(j-1)!} e^{-\theta j} < \infty,
\end{eqnarray*}
where the last inequality follows from \eqref{eqBC} and the following
bound for $j \ge(1-\theta)^{-1}$
\[
\label{} \P( \Gamma_j \le\theta j) = \int_0^{\theta j}
\frac{x^{j-1}}{(j-1)!} e^{-x} \,dx \le\frac{(\theta j)^{j}}{(j-1)!} e^{-\theta j},
\]
which holds because the function under the integral is increasing on
the interval of integration.
Now we observe that
\begin{eqnarray*}
\sum_{j=1}^{\infty} \P\bigl( \bigl\|H(\cdot,
\Gamma_j, V_j)\bigr\| > 1\bigr) &=&\sum
_{j=1}^{\infty} \int_0^{\infty}
\P\bigl( \bigl\|H(\cdot, r, V_j)\bigr\| > 1\bigr) \frac
{r^{j-1}}{(j-1)!}
e^{-r} \,dr
\\
&=& \int_0^{\infty} \P\bigl(\bigl \|H(\cdot, r,
V_j)\bigr\| > 1\bigr) \sum_{j=1}^{\infty}
\frac{r^{j-1}}{(j-1)!} e^{-r} \,dr
\\
&= &\int_0^{\infty} \P\bigl( \bigl\|H(\cdot, r, V)\bigr\| > 1
\bigr) \,dr,
\end{eqnarray*}
which proves \eqref{eqint}. We also notice that \eqref{eqG1} and the
monotonicity of $H$ imply~\eqref{eqG0}.
\end{pf}

\begin{pf*}{Proof of Theorem~\ref{thID1}}
Define a bounded function $H_0$ by
\[
\label{eqH0} H_0(t, r, v) = H(t, r, v) \1\bigl(\bigl\| H(\cdot, r, v) \bigr\|
\le1\bigr),
\]
and let
\[
\label{eqA0} A^{u}_0(t) = \int_0^{u}
\E\bigl\{ H_0(t,r,V)\bigr\} \,dr.
\]
Consider for $u \ge0$,
%
\begin{equation}
\label{eqY0u} Y_0^u(t) = \sum
_{j: \Gamma_j \le u} H_0(t, \Gamma_j,
V_j) - A_0^{u}(t).
\end{equation}
Let $\rho_{t_1,\ldots,t_n}$ be defined by the left-hand side of \eqref
{eqnu} with $H$ replaced by $H_0$, $0\le t_1 <\cdots< t_n \le1$.
$\rho_{t_1,\ldots,t_n}$ is a L\'evy measure on $E^n$ because $\rho_{t_1,\ldots,t_n} \le\nu_{t_1,\ldots,t_n}$, see~\cite{A-G-book},
Chapter~3.4, Exercise~4.
Referring to the proof of Theorem 2.4 in \cite{RosinskiOnSeries}, we
infer that for each $t \in[0,1]$,
\[
Y_0(t)= \lim_{u\to\infty} Y_0^u(t)
\]
exists a.s. Moreover, the finite dimensional distributions of $\{
Y_0(t)\}_{t\in[0,1]}$ are given by \eqref{eqLK1} with $b\equiv0$ and
$\nu_{t_1,\ldots,t_n}$ replaced by $\rho_{t_1,\ldots,t_n}$.

Let
\[
\label{} b_0(t) = b(t) -\int_{0}^{\infty}
\E\bigl[\!\bigl[ H(t, r, V)\bigr]\!\bigr]\1\bigl(\bigl\|H(\cdot ,r,V)\bigr\|>1\bigr) \,dr.
\]
Using Lemma \ref{lemnec} we infer that the above integral is well
defined and $b_0 \in\de$.
In view of \eqref{eqG0}, the process
\[
\label{} Z(t) = b_0(t) + \sum_{j=1}^{\infty}
\bigl\{ H(t, \Gamma_j, V_j) - H_0(t,
\Gamma_j, V_j) \bigr\}
\]
is also well defined, as the series has finitely many terms a.s., and
$Z$ has paths in $D([0,1]; E)$. Processes $Y_0$ and $Z$ are independent
because they depend on a Poisson point process $N= \sum_{j=1}^{\infty}
\delta_{(\Gamma_j, V_j)}$ restricted to disjoint sets $\{(r,v)\dvtx \|
H(\cdot, r, v)\|\le1\}$ and its complement, respectively. Finite
dimensional distributions of $Z-b_0$ are compound Poisson as $(\nu_{t_1,\ldots,t_n} - \rho_{t_1,\ldots,t_n})(E^n) < \infty$ due to
\eqref{eqint}.
We infer that
\[
\label{} Y_0 + Z \eid X,
\]
where the equality holds in the sense of finite dimensional
distributions. Thus $Y_0$ has a modification with paths in $D([0,1];
E)$ a.s.

The family $\{\mathcal{L}(Y_0(t))\}_{t\in[0,1]}$ is relatively
compact because $\mathcal{L}(Y_0(t))$ is a convolution factor of
$\mathcal{L}(X(t))$ and $\{\mathcal{L}(X(t))\}_{t\in[0,1]}$ is
relatively compact; use Theorem~4.5, Chapter~1 together with
Corollary~4.6, Chapter~3 from \cite{A-G-book}. The latter claim follows
from the fact that the function $t \mapsto\mathcal{L}(X(t))$ is \ca.
Since $\rho_t(x\dvtx |x|_E >1) =0$ for all $t \in[0,1]$, $\{|Y_0(t)|_E \dvtx t\in[0,1] \}$ is also uniformly integrable; see
\cite{Jurek-Rosinski}, Theorem 2.

It follows from \eqref{eqY0u} that the $\de$-valued process $\{Y_0^u\}_{u \ge0}$ has independent increments and $\E Y_0^u(t)=0$ for all $t$
and $u$. By Theorem~\ref{thm1}(iii)
%
\begin{equation}
\label{Y0-conv}\bigl \|Y_0^u - Y_0\bigr\| \to0\qquad
\mbox{a.s.}
\end{equation}
as $u=u_n \uparrow\infty$. Since for each $t\in[0,1]$, the process
$\{
Y_0^u(t)\}_{u \ge0}$ is \ca\ \eqref{Y0-conv} holds also for the
continuous parameter $u \in\R_+$, $u \to\infty$; cf. \cite{RosinskiOnSeries}, Lemma~2.3.

Therefore, with probability 1 as $u \to\infty$,
\begin{eqnarray*}
&& \bigl\| Y^u - Y_0 - Z\bigr\| \le\bigl\|Y^u -
Y_0^u - Z\bigr\| + \bigl\|Y_0^u -
Y_0\bigr\|
\\
&&\qquad\le \biggl\|\sum_{j: \Gamma_j > u} \bigl\{H(\cdot,
\Gamma_j, V_j) - H_0(\cdot,
\Gamma_j, V_j)\bigr\} \biggr\|
\\
&&\qquad\quad{} + \biggl\|\int_{u}^{\infty} \E\bigl[\!\bigl[ H(\cdot,
r,V)\bigr]\!\bigr]\1\bigl(\bigl\|H(\cdot ,r,V)\bigr\| >1\bigr) \,dr \biggr\| +
\bigl\|Y_0^u - Y_0\bigr\|
\\
&&\qquad\le \sum_{j: \Gamma_j > u} \bigl\| H(\cdot, \Gamma_j,
V_j)\bigr\|\1\bigl(\bigl\|H(\cdot,\Gamma_j,V_j)\bigr\|>1
\bigr)
\\
&&\qquad\quad{} + \int_{u}^{\infty} \P\bigl(\bigl \|H(\cdot,r,V)\bigr\|>1
\bigr) \,dr +\bigl \|Y_0^u - Y_0\bigr\|\\
&&\qquad =
I_1(u) + I_2(u) + I_3(u) \to0.
\end{eqnarray*}
Indeed, $I_1(u)=0$ for sufficiently large $u$ by \eqref{eqG0}, $I_2(u)
\to0$ by \eqref{eqint} and $I_3(u) \to0$ by \eqref{Y0-conv}.
The proof is complete.
\end{pf*}

\section{Symmetric stable processes with \ca\ paths}\label{sec-4}

In this section we illustrate applications of results of Section \ref
{series-rep} to stable processes. Let $X=\{X(t)\}_{t\in[0,1]}$ be
right-continuous in probability symmetric $\alpha$-stable process,
$\alpha\in(0,2)$. Any such process has a stochastic integral
representation of the form
%
\begin{equation}
\label{int-X} X(t)=\int_S f(t,s) M(ds)\qquad \mbox{a.s. for each
} t \in[0,1],
\end{equation}
where $M$ is an independently scattered symmetric $\alpha$-stable
random measure defined on some measurable space $(S,\mathcal S)$ with a
finite control measure $m$, that is, for all $A\in\mathcal S$
%
\begin{equation}
\label{r-m} \E\exp\bigl\{i \theta M(A)\bigr\} = \exp\bigl\{-\vert \theta
\vert^\alpha m(A)\bigr\},
\end{equation}
and $f(t, \cdot) \in L^{\alpha}(S, m)$ for all $t \in[0,1]$; see
Rajput and Rosi{\'n}ski \cite
{Rosinskispec}, Theorem~5.2, for the almost sure representation in
\eqref{int-X}. Therefore, all symmetric $\alpha$-stable processes are
Volterra processes.
Conversely, a process given by \eqref{int-X} and~\eqref{r-m} is symmetric
$\alpha$-stable.

A trivial case of \eqref{int-X} is when $X$ is a standard symmetric L\'
evy process. In that case, $M$ is a random measure generated by the
increments of $X$, $S=[0,1]$, $m$ is the Lebesgue measure and $f(t,
s)=\1_{(0,t]}(s)$.

A process $X$ given by \eqref{int-X} has many series representations of
the form \eqref{eqY1} because there are many ways to construct a
function $H$ satisfying \eqref{eqnu}; see \cite{Rosinskiseriespoint}.
A~particularly nice representation, called the LePage representation, is
the following.
Let $\{V_j\}$ be an i.i.d. sequence of random elements in $S$ with the
common distribution $m/m(S)$. Let $\{\Gamma_j\}$ be a sequence of
partial sums of standard exponential random variables independent of
the sequence $\{V_j\}$. Let $\{\varepsilon_j\}$ be an i.i.d. sequence of
symmetric Bernoulli random variables. Assume that the random sequences
$\{V_j\}$, $\{\Gamma_j\}$ and $\{\varepsilon_j\}$ are independent. Then
for each $t \in[0,1]$,
%
\begin{equation}
\label{sas} X(t)= c_{\alpha} m(S)^{1/\alpha} \sum
_{j=1}^{\infty} \varepsilon_j
\Gamma_j^{-1/\alpha} f(t,V_j)\qquad \mbox{a.s.}
\end{equation}
(the almost sure representation is obtained by combining
\cite{Rosinskiseriespoint} and \cite{Rosinskisumrep}, Proposition~2).
Here $c_{\alpha}= [-\alpha\cos(\pi\alpha/2)\Gamma(-\alpha
)]^{-1/\alpha
}$ for $\alpha\neq1$ and $c_1=2/\pi$.

\begin{corollary}\label{cor-stable0}
Let $X=\{X(t)\}_{t\in[0,1]}$ be a symmetric $\alpha$-stable process of
the form \eqref{int-X}, where $\alpha\in(0,2)$. Assume that $X$ is
\ca\
and continuous in probability and also that $f(\cdot, s) \in D[0,1]$
for all $s$. Then with probability~1,
\[
\label{sas1} X(t)= c_{\alpha} m(S)^{1/\alpha} \sum
_{j=1}^{\infty} \varepsilon_j
\Gamma_j^{-1/\alpha} f(t,V_j)\qquad \mbox{for all } t
\in[0,1],
\]
where the series converges a.s. uniformly on $[0,1]$. Therefore, with
probability~1
%
\begin{equation}
\label{jump1} \Delta X(t)= c_{\alpha} m(S)^{1/\alpha} \sum
_{j=1}^{\infty} \varepsilon_j
\Gamma_j^{-1/\alpha} \Delta f(t,V_j),\qquad t \in[0,1],
\end{equation}
where the series has no more than one nonzero term for each $t$. That
is,
%
\begin{equation}
\label{eq-jumps} \P \bigl( \Delta f(t,V_j)\Delta f(t,V_k)
= 0 \mbox{ for all } j\ne k \mbox{ and } t\in[0,1] \bigr)=1.
\end{equation}
\end{corollary}
\begin{pf}
In view of Corollary \ref{corseries} we only need to show \eqref
{eq-jumps}. $f(\cdot,V_j)$ are i.i.d. \ca\ processes. Since $X$ is
continuous in probability, from \eqref{sas} by a symmetrization
inequality, we get $\P(\Delta f(t,V_j)=0)=1$ for each $t\in[0,1]$.
Thus for each $j\ne k$ and $\mu=\mathcal{L} (f(\cdot, V_k)
)$ we have
\begin{eqnarray*}
\label{} && \P \Bigl( \sup_{1\le t \le1} \bigl|\Delta f(t,V_j)\Delta
f(t,V_k)\bigr| = 0 \Bigr)
\\
&&\qquad = \int_{D[0,1]} \P \Bigl( \sup_{1\le t \le1} \bigl|\Delta
f(t,V_j)\Delta x(t)\bigr| = 0 \Bigr) \mu(dx) = 1,
\end{eqnarray*}
because $\Delta x(t) \ne0$ for at most countably many $t$. This
implies \eqref{eq-jumps}.
\end{pf}

Next we consider some functionals of the jump process $\Delta X$. Let
$V_p(g)$ be defined as
\[
\label{} V_p(g)=\sum_{t\in[0,1]}\bigl\vert\Delta
g(t)\bigr\vert^p,
\]
where $g \in D[0,1]$ and $p>0$. Recall that a random variable $Z$ is
Fr\'echet distributed with shape parameter $\alpha>0$ and scale
parameter $\sigma>0$ if for all $x>0$, $\P(Z\leq x)=e^{-(x/\sigma
)^{-\alpha}}$. The results below are well known for a L\'evy stable
process. Below we give their versions for general \ca\ symmetric stable
processes.

\begin{corollary} \label{cor-stable}
Under the assumptions of Corollary \ref{cor-stable0} we have the following:
\begin{longlist}[(iii)]
\item[(i)]
$V_p(X)<\infty$ a.s.  if and only if either
$f(\cdot,s)$ is continuous for $m$-a.a. $s$, in which case $V_p(X)=0$
a.s.  or $p>\alpha$ and $\int V_p(f(\cdot,s))^{\alpha/p} m(ds) \in
(0,\infty)$. In the latter case, $V_p(X)$ is a positive $(\alpha
/p)$-stable random variable with shift parameter~$0$ and scale parameter
\[
\label{} c_{\alpha}^pc_{\alpha/p}^{-1} \biggl(
\int V_p\bigl(f(\cdot,s)\bigr)^{\alpha/p} m(ds)
\biggr)^{p/\alpha}.
\]
\item[(ii)]
The largest jump of $X$ in absolute value, $\sup_{t\in[0,1]}\abs{\Delta X(t)}$, is Fr\'echet distributed with shape
parameter $\alpha$ and scale parameter
\[
\label{} c_\alpha \biggl(\int\sup_{t\in[0,1]}\bigl\vert\Delta f(t,s)
\bigr\vert^\alpha m(ds) \biggr)^{1/\alpha}.
\]
\item[(iii)]
The largest jump of $X$, $\sup_{t\in[0,1]}
\Delta
X(t)$, is
Fr\'echet distributed with shape parameter $\alpha$ and scale parameter
\begin{eqnarray*}
\label{def-u} &&\frac{c_\alpha}{2} \biggl[ \biggl(\int\Bigl\vert\sup_{t\in[0,1]}
\Delta f(t,s)\Bigr\vert^\alpha m(ds) \biggr)^{1/\alpha}
\\
&&\qquad{}+ \biggl(\int\Bigl\vert
\inf_{t\in[0,1]} \Delta f(t,s)\Bigr\vert^\alpha m(ds) \biggr)^{1/\alpha}
\biggr].
\end{eqnarray*}
\end{longlist}

\end{corollary}

\begin{pf}
(i): By \eqref{jump1} and \eqref{eq-jumps} we have a.s. %
\begin{eqnarray*}
\label{ty} \sum_{t\in[0,1]} \bigl\vert\Delta X(t)
\bigr\vert^p &=& c_{\alpha}^p m(S)^{p/\alpha} \sum
_{t\in[0,1]} \sum_{j=1}^\infty
\Gamma_j^{-p/\alpha}\bigl|\Delta f(t,V_j)\bigr|^p
\\
& =&c_{\alpha}^p m(S)^{p/\alpha} \sum
_{j=1}^\infty\Gamma_j^{-1/(\alpha
/p)}
V_p\bigl(f(\cdot,V_j)\bigr),
\end{eqnarray*}
which show (i); see, for example,  \cite{Stable}.

(ii): By (\ref{jump1}) and (\ref{eq-jumps}) we have a.s.
\begin{eqnarray*}
\sup_{t\in[0,1]} \bigl|\Delta X(t)\bigr| &=& c_{\alpha} m(S)^{1/\alpha}
\sup_{t\in[0,1]} \sup_{j \in\N} \Gamma_j^{-1/\alpha}\bigl|
\Delta f(t, V_j)\bigr|
\\
&=& c_{\alpha} m(S)^{1/\alpha} \sup_{j \in\N}
\Gamma_j^{-1/\alpha
} W_j,
\end{eqnarray*}
where $W_j= \sup_{t\in[0,1]} |\Delta f(t, V_j)|$ are i.i.d. random
variables. For $j\in\N$ set $\xi_j = \Gamma_j^{-1/\alpha} W_j$. Then
$\sum_{j=1}^{\infty} \delta_{\xi_j}$ is a Poisson point process on
$\R_+$ with the intensity measure $\mu(dx)= \alpha\E W_1^{\alpha}
x^{-\alpha-1} \,dx$, $x>0$. Let $\eta_j= (\E W_1^{\alpha})^{1/\alpha
}\Gamma_j^{-1/\alpha}$. Since the Poisson point processes $\sum_{j=1}^{\infty} \delta_{\xi_j}$ and $\sum_{j=1}^{\infty} \delta_{\eta
_j}$ have the same intensity measures, the distributions of their
measurable functionals are equal. That is, $\sup_{j} \xi_j \eid\sup_{j} \eta_j$, so that
\begin{eqnarray*}
\label{} \sup_{t\in[0,1]} \bigl|\Delta X(t)\bigr| & \eid& c_{\alpha}
m(S)^{1/\alpha} \sup_{j \in\N} \bigl(\E W_1^{\alpha}
\bigr)^{1/\alpha}\Gamma_j^{-1/\alpha}
\\
&=& c_{\alpha} m(S)^{1/\alpha} \bigl(\E W_1^{\alpha}
\bigr)^{1/\alpha} \Gamma_1^{-1/\alpha}.
\end{eqnarray*}
This shows (ii).

(iii): By \eqref{jump1} and \eqref{eq-jumps} we have a.s. %
\begin{eqnarray*}
\sup_{t\in[0,1]} \Delta X(t) &=& c_{\alpha} m(S)^{1/\alpha}
\sup_{t\in
[0,1]} \sup_{j \in\N} \varepsilon_j
\Gamma_j^{-1/\alpha} \Delta f(t, V_j)
\\
&=& c_{\alpha} m(S)^{1/\alpha} \sup_{j \in\N}
\Gamma_j^{-1/\alpha
} W_j,
\end{eqnarray*}
where
\[
\label{}
W_j = \cases{ \displaystyle\sup_{t\in[0,1]} \Delta f(t,
V_j), &\quad $\mbox{if } \varepsilon_j=1,$ \vspace*{2pt}
\cr
\displaystyle -
\inf_{t\in[0,1]} \Delta f(t, V_j), &\quad $\mbox{if }
\varepsilon_j=-1.$ }
\]
Observe that $W_j\geq0$ is an i.i.d. sequence. Proceeding as in (ii)
we get
\[
\label{} \sup_{t\in[0,1]} \Delta X(t) \eid c_{\alpha}
m(S)^{1/\alpha} \bigl(\E W_1^{\alpha}\bigr)^{1/\alpha}
\Gamma_1^{-1/\alpha},
\]
which completes the proof.
\end{pf}

It can be instructive to examine how Corollaries \ref
{cor-stable0} and \ref
{cor-stable} apply to the above mentioned standard symmetric stable L\'
evy process.

The crucial assumption in the above corollaries is that a stable
process has \ca\ paths. To this end we establish a sufficient criterion
which extends a recent result of Davydov and Dombry \cite
{LePage-series} obtained by
different methods; see Remark \ref{DD}.

\begin{theorem}\label{pro-ca}
Let $X=\{X(t)\}_{t\in[0,1]}$ be given by \eqref{int-X} and let
$\alpha
\in(1,2)$.
Assume that there exist $\beta_1,\beta_2>1/2$, $p_1 > \alpha$,
$p_2>\alpha/2$ and increasing continuous functions
${F_1,F_2}\dvtx[0,1]\to\R$
such that for all $0\leq t_1\leq t\leq t_2\leq1$,
%
\begin{eqnarray}
\label{B1}  &&\int\bigl\vert f(t_2,s)-f(t_1,s)
\bigr\vert^{p_1} m(ds)\leq \bigl|F_1(t_2)-F_1(t_1)\bigr|^{\beta_1},
\\
\label{B2}  &&\int\bigl\vert\bigl(f(t,s)-f(t_1,s)\bigr)
\bigl(f(t_2,s)-f(t,s)\bigr)\bigr\vert^{p_2} m(ds)
\nonumber
\\[-8pt]
\\[-8pt]
\nonumber
&&\qquad\leq
\bigl|F_2(t_2)-F_2(t_1)\bigr|^{2\beta_2}.
\end{eqnarray}
Then $X$ has a \ca\ modification.
\end{theorem}

\begin{pf}
Decompose $M$ as $M=N+N'$, where $N$ and $N'$ are independent,
independently scattered random measures given by
%
\begin{equation}
\label{def-N} \E\exp\bigl\{i \theta N(A)\bigr\} = \exp \biggl\{ k_\alpha
m(A) \int_{0}^1 \bigl( \cos(\theta x)-1 \bigr)
x^{-1-\alpha} \,dx \biggr\}
\end{equation}
and
\[
\E\exp\bigl\{i \theta N'(A)\bigr\} = \exp \biggl\{
k_\alpha m(A) \int_{1}^{\infty} \bigl( \cos(
\theta x)-1 \bigr) x^{-1-\alpha} \,dx \biggr\},
\]
where $A\in\mathcal S$ and $k_{\alpha}= \alpha c_{\alpha}^{\alpha}$.
Treating $f=\{f(t,\cdot)\}_{t\in[0,1]}$ as a stochastic process
defined on $(S, m/m(S))$, observe that by \cite{Billingsley},
Theorem~13.6, \eqref{B1}--\eqref{B2} imply that $f$ has a modification
with paths in $D[0,1]$. Therefore, without affecting \eqref{int-X}, we
may choose $f$ such that $t\mapsto f(t,s)$ is \ca\ for all $s$. Since
$N'$ has finite support a.s. [$N'(S)$ has a compound Poisson
distribution], it suffices to show that a process $Y=\{Y(t)\}_{t\in
[0,1]}$ given by
\[
\label{} Y(t)=\int_S f(t,s) N(ds),
\]
has a \ca\ modification.
To this end, invoking again \cite{Billingsley}, Theorem~13.6, it is
enough to show that $Y$ is right-continuous in probability and there
exist a continuous increasing function
${F}\dvtx[0,1]\to\R$, $\beta
>\frac{1}{2}$ and $p>0$ such that for all $0\leq t_1\leq t\leq t_2\leq
1$ and $\lambda\in(0,1)$
%
\begin{equation}
\label{enough} \P \bigl(\bigl\vert Y(t)-Y(t_1)\bigr\vert\wedge
\bigl|Y(t_2)-Y(t)\bigr|>\lambda \bigr)\leq \lambda^{-p}
\bigl[F(t_2)-F(t_1)\bigr]^{2\beta}.
\end{equation}
(Notice that \cite{Billingsley}, Theorem~13.6 assumes that \eqref
{enough} holds for all $\lambda>0$, but the proof reveals that
$\lambda
\in(0,1)$ suffices.)

Set
\[
\label{} Z_1=Y(t)-Y(t_1) = \int_{S}
h_1 \,dN \quad\mbox{and}\quad Z_2=Y(t_2)-Y(t) = \int
_{S} h_2 \,dN,
\]
where $h_1(s)= f(t,s) -f(t_1,s)$ and $h_2(s)= f(t_2,s) -f(t,s)$. Below
$C$ will stand for a constant that is independent of $\lambda, t_1, t,
t_2$ but may be different from line to line. Applying \eqref
{est-lemma-2} of Lemma~\ref{est-Z} and assumptions \eqref{B1}--\eqref
{B2} we get
\begin{eqnarray*}
&& \P \bigl( \bigl\vert Y(t)-Y(t_1)\bigr\vert \wedge\bigl\vert Y(t_2)-Y(t)
\bigr\vert>\lambda \bigr)
\\
& &\qquad= \P\bigl(\vert Z_1\vert\wedge\vert Z_2\vert>\lambda\bigr) \le\P
\bigl(\vert Z_1 Z_2\vert>\lambda^2\bigr)
\\
& &\qquad\le C \biggl(\lambda^{-2p_1} \int\vert h_1
\vert^{p_1} \,dm \int\vert h_2\vert^{p_1} \,dm +
\lambda^{-2p_2}\int \vert h_1 h_2
\vert^{p_2} \,dm \biggr)
\\
&&\qquad \le C \bigl(\lambda^{-2p_1} \bigl|F_1(t_2)-F_1(t_1)\bigr|^{2\beta_1}
+ \lambda^{-2p_2}\bigl|F_2(t_2)-F_2(t_1)\bigr|^{2\beta_2}
\bigr).
\end{eqnarray*}
Thus \eqref{enough} holds for $\lambda\in(0,1)$ with $p=2(p_1\vee
p_2)$, $\beta=\beta_1\wedge\beta_2$ and $F= C (F_1+F_2)$. The last
bound in Lemma \ref{est-Z} together with \eqref{B1} imply continuity of
$Y$ in $L^{p_1}$. The proof will be complete after proving the
following lemma.
\end{pf}

\begin{lemma}\label{est-Z}
Let $N$ be given by \eqref{def-N} and let $Z_k=\int_{S} h_k\, dN$, where
$h_k$ is a deterministic function integrable with respect to $N$, $k=1,2$.
For all $p_1> \alpha$ and $p_2>\alpha/2$ there exists a constant $C>0$,
depending only on $p_1,p_2$ and $\alpha$, such that for all $\lambda>0$
%
\begin{eqnarray}
\label{est-lemma-2}\quad  &&\P\bigl(\vert Z_1Z_2\vert>\lambda\bigr)
\nonumber
\\[-8pt]
\\[-8pt]
\nonumber
&&\qquad\leq C
\biggl(\lambda^{-p_1} \int\vert h_1\vert^{p_1} \,dm
\int\vert h_2\vert^{p_1} \,dm + \lambda^{-p_2}\int
\vert h_1 h_2\vert^{p_2} \,dm \biggr).
\end{eqnarray}
Moreover, $E |Z_1|^{p_1} \le C \int|h_1|^{p_1} \,dm$.
\end{lemma}

\begin{pf}
To show \eqref{est-lemma-2} we may and do assume that $h_1$ and $h_2$
are simple functions of the form $h_1=\sum_{j=1}^n a_j \1_{A_j}$ and
$h_2=\sum_{j=1}^n b_j \1_{A_j}$, where $(A_j)_{j=1}^n$ are disjoint
measurable sets and $(a_j)_{j=1}^n, (b_j)_{j=1}^n\subseteq\R$.
We have
\[
\label{} Z_1 Z_2 = \sum_{j,k=1: k\neq j}^n
a_j b_k N(A_j) N(A_k)+\sum
_{k=1}^n a_k b_k
N(A_k)^2=T+D,
\]
and hence
%
\begin{equation}
\label{eq-min} \P\bigl(\vert Z_1Z_2\vert>\lambda\bigr) \leq\P\bigl(\vert
T\vert >\lambda/2\bigr)+ \P\bigl(\vert {D}\vert>\lambda /2\bigr),\qquad \lambda>0.
\end{equation}

For $(u_j)_{j=1}^n\subseteq\R$ set $\mathbf{X}=(u_1 N(A_1),\ldots, u_n
N(A_n))$ and $h=\sum_{j=1}^n u_j \1_{A_j}$. The Euclidean norm on $\R^n$ is denoted $\abs{\mathbf{x}}_n=(\sum_{j=1}^n x_j^2)^{1/2}$. We
claim that for all $p> \alpha$ there exists a constant $C_{1}$, only
depending on $p$, $\alpha$ and $m(S)$, such that
%
\begin{eqnarray}
\label{est-X} \E\abs{\mathbf{X}}_n^p&\leq&
C_{1}\int\abs{h}^{p} \,dm,
\\
\label{est-p} \E\biggl\vert\int_{S} h(s) N(ds)
\biggr\vert^p &\leq& C_{1}\int\abs{h}^{p} \,dm.
\end{eqnarray}
We will show \eqref{est-X} and \eqref{est-p} at the end of this proof. Now
we notice that for $p_2>\alpha/2$ and $u_j=\abs{a_j b_j}^{1/2}$,
$j=1,\ldots,n$ bound \eqref{est-X} yields
%
\begin{equation}
\label{est-D}\qquad \E\abs{D}^{p_2} \leq\E\abs{\mathbf{X}}_n^{2p_2}
\leq C_{1} \int \bigl(\vert h_1 h_2
\vert^{1/2} \bigr)^{2p_2} \,dm=C_{1} \int
\abs{h_1 h_2}^{p_2} \,dm.
\end{equation}
Now let $p_1>\alpha$. By a decoupling inequality (see \cite{KwapienRS},
Theorem~6.3.1), there exists a constant $C_2$, only depending on $p$,
such that
\[
\label{} \E\abs{T}^{p_1} \leq C_2 \int_{\Omega}
\E \biggl(\biggl\vert\int_{S} \phi\bigl(s,\omega'
\bigr) N(ds)\biggr\vert^{p_1} \biggr) \P\bigl(d\omega'\bigr),
\]
where $\phi(s,\omega')=\sum_{j=1}^n \tilde a_j(\omega') \1_{A_j}(s)$
and $\tilde a_j(\omega')=a_j\sum_{k=1: k\neq j}^n b_k N(A_k)(\omega')$.
By~\eqref{est-p} we have
\[
\label{} \E\biggl\vert\int_S \phi\bigl(s,\omega'
\bigr) N(ds)\biggr\vert^{p_1} \leq C_{1} \sum
_{j=1}^n \abs{a_j}^{p_1}
m(A_j) \Biggl\vert\sum_{k=1: k\neq j}^n
b_k N(A_k) \bigl(\omega'\bigr)
\Biggr\vert^{p_1},
\]
and hence by another application of \eqref{est-p},
%
\begin{equation}
\label{est-I23} \E\abs{T}^{p_1} \leq C_1^2
C_2 \int\abs{h_1}^{p_1} \,dm\int\abs
{h_2}^{p_1} \,dm.
\end{equation}
Combining \eqref{eq-min}, \eqref{est-D} and \eqref{est-I23} with
Markov's inequality we get \eqref{est-lemma-2}.

To show \eqref{est-X} we use Rosi\'nski and Turner \cite
{Rosinski-Turner}. Notice that the
L\'evy measure of $\mathbf{X}$ is given by
%
\begin{eqnarray}
\label{def-nu}\quad  \nu(B)=\ \frac{1}{2} k_\alpha\int
_{-1}^1 \biggl(\int_{\R^n}
\1_{B}(r\theta) \kappa (d\theta) \biggr) |r|^{-1-\alpha} \,dr\qquad  B \in
\mathcal{B}\bigl(\R^n\bigr),
\end{eqnarray}
where $\kappa=\sum_{j=1}^n m(A_j)\delta_{u_j \mathbf{e}_j}$, and
$(\mathbf{e}_j)_{j=1}^n$ is the standard basis in $\R^n$.
For all $l>0$ set
\begin{eqnarray*}
\label{} \xi_p(l) &= & \int_{\R^n} \bigl\vert x
l^{-1}\bigr\vert_n^p\1_{\{\vert x
l^{-1}\vert_n>1\}
} \nu(dx)+ \int
_{\R^n} \bigl\vert x l^{-1}\bigr\vert_n^2
\1_{\{\vert x
l^{-1}\vert_n\leq1\}
} \nu(dx)
\\
&=& V_1(l)+V_2(l).
\end{eqnarray*}
According to \cite{Rosinski-Turner}, Theorem 4, $c_p l_p \le(\E\abs
{\mathbf{X}}_n^p)^{1/p} \le C_p l_p$ for some constants $c_p, C_p$
depending only on $p$, where $l=l_p$ is the unique solution of the
equation $\xi_p(l)=1$. From the above decomposition we have either
$V_1(l_p) \ge1/2$ or $V_2(l_p) \ge1/2$. In the first case
\[
\frac{1}{2} \le V_1(l_p) \le\int
_{\R^n} \bigl\vert x l_p^{-1}
\bigr\vert_n^p \nu (dx) = C_{3}
l_p^{-p} \int\abs{h}^{p} \,dm,
\]
where $C_{3}=k_\alpha/(p-\alpha)$. Thus
\[
\label{} \E\abs{\mathbf{X}}_n^p \le2C_p^p
C_3 \int\abs{h}^{p} \,dm,
\]
proving \eqref{est-X}. If $V_2(l_p) \ge1/2$, then we consider two
cases. First assume that $p\in(\alpha,2]$. We have
\[
\frac{1}{2} \le V_2(l_p) \le\int
_{\R^n} \bigl\vert x l_p^{-1}
\bigr\vert_n^p \nu (dx) = C_3 l_p^{-p}
\int\abs{h}^{p} \,dm,
\]
which yields \eqref{est-X} as above. Now we assume that $p>2$. We get
\[
\frac{1}{2} \le V_2(l_p) \le\int
_{\R^n} \bigl\vert x l_p^{-1}
\bigr\vert_n^2 \nu (dx) = C_4 l_p^{-2}
\int\abs{h}^{2} \,dm,
\]
where $C_4=k_\alpha/(2-\alpha)$. Applying Jensen's inequality to the
last term we get
\[
\frac{1}{2} \le C_4 m(S)^{1-2/p} l_p^{-2}
\biggl(\int\vert h\vert^p \,dm \biggr)^{2/p},
\]
which yields the desired bound for $l_p$, establishing \eqref{est-X}
for all $p>\alpha$. The proof of \eqref{est-p} is similar, and it is
therefore omitted. This completes the proof of the lemma.
\end{pf}
%

\begin{remark}\label{DD}
In a recent paper Davydov and Dombry \cite{LePage-series} obtained
sufficient conditions
for the uniform convergence in $D[0,1]$ of the LePage series \eqref
{sas}, which in turn yield criteria for a symmetric stable process to
have \ca\ paths. Their result is a special case of our Theorem \ref
{pro-ca} combined with Corollary \ref{cor-stable0}, when one takes
$p_1=p_2=2$ and assumes additionally that $\E\norm{f(\cdot,
V)}^\alpha
<\infty$. The methods are also different from ours.

In our approach, we established the existence of a \ca\ version first,
using special distributional properties of the process. Then the
uniform convergence of the LePage series, and also of other shot noise
series expansions, follows automatically by Corollary \ref{corsym}.
This strategy applies to other infinitely divisible processes as well.
Here we provided only an example of possible applications of the
results of Section \ref{series-rep}.
\end{remark}

\section*{Acknowledgments}
The authors are grateful to J\o rgen Hoffmann-J\o rgensen for
discussions and interest in this work and to the anonymous referee for
careful reading of the manuscript and helpful suggestions.

%

\printaddresses


\begin{thebibliography}{23}

\bibitem{A-G-book}
%
\begin{bbook}[mr]
\bauthor{\bsnm{Araujo},~\bfnm{Aloisio}\binits{A.}} \AND
\bauthor{\bsnm{Gin{\'e}},~\bfnm{Evarist}\binits{E.}}
(\byear{1980}).
\btitle{The Central Limit Theorem for Real and {B}anach Valued Random
Variables}.
\bpublisher{Wiley}, \blocation{New York-Chichester-Brisbane}.
\bid{mr={0576407}}
\bptok{imsref}%
\end{bbook}
%
\endbibitem

\bibitem{Billingsley}
%
\begin{bbook}[mr]
\bauthor{\bsnm{Billingsley},~\bfnm{Patrick}\binits{P.}}
(\byear{1999}).
\btitle{Convergence of Probability Measures},
\bedition{2nd} ed.
\bpublisher{Wiley}, \blocation{New York}.
\bid{doi={10.1002/9780470316962}, mr={1700749}}
\bptok{imsref}%
\end{bbook}
%
\endbibitem

\bibitem{BR88}
%
\begin{barticle}[mr]
\bauthor{\bsnm{Byczkowski},~\bfnm{T.}\binits{T.}} \AND
\bauthor{\bsnm{Ryznar},~\bfnm{M.}\binits{M.}}
(\byear{1988}).
\btitle{Series of independent v
ector-valued random variables and absolute continuity of seminorms}.
\bjournal{Math. Scand.}
\bvolume{62}
\bpages{59--74}.
\bid{mr={0961583}}
\bptok{imsref}%
\end{barticle}
%
\endbibitem


\bibitem{Taylor-LLN-D-space}
%
\begin{barticle}[mr]
\bauthor{\bsnm{Daffer},~\bfnm{Peter~Z.}\binits{P.~Z.}} \AND
\bauthor{\bsnm{Taylor},~\bfnm{Robert~L.}\binits{R.~L.}}
(\byear{1979}).
\btitle{Laws of large numbers for {$D[0,1]$}}.
\bjournal{Ann. Probab.}
\bvolume{7}
\bpages{85--95}.
\bid{issn={0091-1798}, mr={0515815}}
\bptok{imsref}%
\end{barticle}
%
\endbibitem

\bibitem{LePage-series}
%
\begin{barticle}[auto:STB|2012/12/28|09:56:27]
\bauthor{\bsnm{Davydov},~\bfnm{Y.}\binits{Y.}} \AND
\bauthor{\bsnm{Dombry},~\bfnm{C.}\binits{C.}}
(\byear{2012}).
\btitle{On the convergence of {LePage} series in Skorokhod
space.}
\bjournal{Statist. Probab. Lett.}
\bvolume{82}
\bpages{145--150}.
\bid{mr={2863036}}
\bptok{imsref}%
\end{barticle}
%
\endbibitem


\bibitem{JHJBook}
%
\begin{bbook}[mr]
\bauthor{\bsnm{Hoffmann-J{\o}rgensen},~\bfnm{J.}\binits{J.}}
(\byear{1994}).
\btitle{Probability with a View Toward Statistics. {V}ol.~{I}}.
\bpublisher{Chapman \& Hall}, \blocation{New York}.
\bid{mr={1278485}}
\bptok{imsref}%
\end{bbook}
%
\endbibitem

\bibitem{Ikeda-T}
%
\begin{barticle}[mr]
\bauthor{\bsnm{Ikeda},~\bfnm{Nobuyuki}\binits{N.}} \AND
\bauthor{\bsnm{Taniguchi},~\bfnm{Setsuo}\binits{S.}}
(\byear{2010}).
\btitle{The {I}t\^o--{N}isio theorem, quadratic {W}iener functionals, and
1-solitons}.
\bjournal{Stochastic Process. Appl.}
\bvolume{120}
\bpages{605--621}.
\bid{doi={10.1016/j.spa.2010.01.009}, issn={0304-4149}, mr={2603056}}
\bptok{imsref}%
\end{barticle}
%
\endbibitem

\bibitem{Ito-Nisio-theorem}
%
\begin{barticle}[mr]
\bauthor{\bsnm{It{\^o}},~\bfnm{Kiyosi}\binits{K.}} \AND
\bauthor{\bsnm{Nisio},~\bfnm{Makiko}\binits{M.}}
(\byear{1968}).
\btitle{On the convergence of sums of independent {B}anach space
valued random
variables}.
\bjournal{Osaka J. Math.}
\bvolume{5}
\bpages{35--48}.
\bid{issn={0030-6126}, mr={0235593}}
\bptok{imsref}%
\end{barticle}
%
\endbibitem

\bibitem{GauQua}
%
\begin{barticle}[mr]
\bauthor{\bsnm{Jain},~\bfnm{Naresh~C.}\binits{N.~C.}} \AND
\bauthor{\bsnm{Monrad},~\bfnm{Ditlev}\binits{D.}}
(\byear{1982}).
\btitle{Gaussian quasimartingales}.
\bjournal{Z. Wahrsch. Verw. Gebiete}
\bvolume{59}
\bpages{139--159}.
\bid{doi={10.1007/BF00531739}, issn={0044-3719}, mr={0650607}}
\bptok{imsref}%
\end{barticle}
%
\endbibitem

\bibitem{JainMonradBp}
%
\begin{barticle}[mr]
\bauthor{\bsnm{Jain},~\bfnm{Naresh~C.}\binits{N.~C.}} \AND
\bauthor{\bsnm{Monrad},~\bfnm{Ditlev}\binits{D.}}
(\byear{1983}).
\btitle{Gaussian measures in {$B_{p}$}}.
\bjournal{Ann. Probab.}
\bvolume{11}
\bpages{46--57}.
\bid{issn={0091-1798}, mr={0682800}}
\bptok{imsref}%
\end{barticle}
%
\endbibitem

\bibitem{Jurek-Rosinski}
%
\begin{barticle}[mr]
\bauthor{\bsnm{Jurek},~\bfnm{Z.~J.}\binits{Z.~J.}} \AND
\bauthor{\bsnm{Rosi{\'n}ski},~\bfnm{J.}\binits{J.}}
(\byear{1988}).
\btitle{Continuity of certain random integral mappings and the uniform
integrability of infinitely divisible measures}.
\bjournal{Teor. Veroyatn. Primen.}
\bvolume{33}
\bpages{560--572}.
\bid{doi={10.1137/1133077}, issn={0040-361X}, mr={0968401}}
\bptok{imsref}%
\end{barticle}
%
\endbibitem

\bibitem{Kallenbergthm}
%
\begin{barticle}[mr]
\bauthor{\bsnm{Kallenberg},~\bfnm{Olav}\binits{O.}}
(\byear{1974}).
\btitle{Series of random processes without discontinuities of the
second kind}.
\bjournal{Ann. Probab.}
\bvolume{2}
\bpages{729--737}.
\bid{mr={0370679}}
\bptok{imsref}%
\end{barticle}
%
\endbibitem

\bibitem{Kallenberg}
%
\begin{bbook}[mr]
\bauthor{\bsnm{Kallenberg},~\bfnm{Olav}\binits{O.}}
(\byear{2002}).
\btitle{Foundations of Modern Probability},
\bedition{2nd} ed.
\bpublisher{Springer}, \blocation{New York}.
\bid{mr={1876169}}
\bptok{imsref}%
\end{bbook}
%
\endbibitem

\bibitem{KwapienRS}
%
\begin{bbook}[mr]
\bauthor{\bsnm{Kwapie{\'n}},~\bfnm{Stanis{\l}aw}\binits{S.}} \AND
\bauthor{\bsnm{Woyczy{\'n}ski},~\bfnm{Wojbor~A.}\binits{W.~A.}}
(\byear{1992}).
\btitle{Random Series and Stochastic Integrals: Single and Multiple}.
\bpublisher{Birkh\"auser}, \blocation{Boston, MA}.
\bid{doi={10.1007/978-1-4612-0425-1}, mr={1167198}}
\bptok{imsref}%
\end{bbook}
%
\endbibitem

\bibitem{Talagrand}
%
\begin{bbook}[mr]
\bauthor{\bsnm{Ledoux},~\bfnm{Michel}\binits{M.}} \AND
\bauthor{\bsnm{Talagrand},~\bfnm{Michel}\binits{M.}}
(\byear{1991}).
\btitle{Probability in {B}anach Spaces: Isoperimetry and Processes}.
\bseries{Ergebnisse der Mathematik und Ihrer Grenzgebiete (3)}
\bvolume{23}.
\bpublisher{Springer}, \blocation{Berlin}.
\bid{mr={1102015}}
\bptok{imsref}%
\end{bbook}
%
\endbibitem

\bibitem{Linde}
%
\begin{bbook}[mr]
\bauthor{\bsnm{Linde},~\bfnm{Werner}\binits{W.}}
(\byear{1986}).
\btitle{Probability in {B}anach Spaces---Stable and Infinitely Divisible
Distributions},
\bedition{2nd} ed.
\bpublisher{Wiley}, \blocation{Chichester}.
\bid{mr={0874529}}
\bptok{imsref}%
\end{bbook}
%
\endbibitem

\bibitem{Rosinskispec}
%
\begin{barticle}[mr]
\bauthor{\bsnm{Rajput},~\bfnm{Balram~S.}\binits{B.~S.}} \AND
\bauthor{\bsnm{Rosi{\'n}ski},~\bfnm{Jan}\binits{J.}}
(\byear{1989}).
\btitle{Spectral representations of infinitely divisible processes}.
\bjournal{Probab. Theory Related Fields}
\bvolume{82}
\bpages{451--487}.
\bid{doi={10.1007/BF00339998}, issn={0178-8051}, mr={1001524}}
\bptok{imsref}%
\end{barticle}
%
\endbibitem

\bibitem{RosinskiDissertationes}
%
\begin{barticle}[mr]
\bauthor{\bsnm{Rosi{\'n}ski},~\bfnm{Jan}\binits{J.}}
(\byear{1987}).
\btitle{Bilinear random integrals}.
\bjournal{Dissertationes Math. (Rozprawy Mat.)}
\bvolume{259}
\bpages{71}.
\bid{issn={0012-3862}, mr={0888463}}
\bptok{imsref}%
\end{barticle}
%
\endbibitem

\bibitem{Rosinskisumrep}
%
\begin{barticle}[mr]
\bauthor{\bsnm{Rosi{\'n}ski},~\bfnm{Jan}\binits{J.}}
(\byear{1989}).
\btitle{On path properties of certain infinitely divisible processes}.
\bjournal{Stochastic Process. Appl.}
\bvolume{33}
\bpages{73--87}.
\bid{doi={10.1016/0304-4149(89)90067-7}, issn={0304-4149}, mr={1027109}}
\bptok{imsref}%
\end{barticle}
%
\endbibitem

\bibitem{RosinskiOnSeries}
%
\begin{barticle}[mr]
\bauthor{\bsnm{Rosi{\'n}ski},~\bfnm{Jan}\binits{J.}}
(\byear{1990}).
\btitle{On series representations of infinitely divisible random vectors}.
\bjournal{Ann. Probab.}
\bvolume{18}
\bpages{405--430}.
\bid{issn={0091-1798}, mr={1043955}}
\bptok{imsref}%
\end{barticle}
%
\endbibitem

\bibitem{Rosinskiseriespoint}
%
\begin{bincollection}[mr]
\bauthor{\bsnm{Rosi{\'n}ski},~\bfnm{Jan}\binits{J.}}
(\byear{2001}).
\btitle{Series representations of {L}\'evy processes from the
perspective of
point processes}.
In \bbooktitle{L\'evy Processes}
\bpages{401--415}.
\bpublisher{Birkh\"auser}, \blocation{Boston, MA}.
\bid{mr={1833707}}
\bptok{imsref}%
\end{bincollection}
%
\endbibitem

\bibitem{Rosinski-Turner}
%
\begin{bmisc}[auto:STB|2012/12/28|09:56:27]
\bauthor{\bsnm{Rosi{\'n}ski},~\bfnm{J.}\binits{J.}} \AND
\bauthor{\bsnm{Turner},~\bfnm{M.}\binits{M.}}
(\byear{2011}).
\bhowpublished{Explicit {$L^p$-norm} estimates of infinitely divisible
random vectors in Hilbert spaces. Unpublished manuscript}.
\bptok{imsref}%
\end{bmisc}
%
\endbibitem\

\bibitem{Stable}
%
\begin{bbook}[mr]
\bauthor{\bsnm{Samorodnitsky},~\bfnm{Gennady}\binits{G.}} \AND
\bauthor{\bsnm{Taqqu},~\bfnm{Murad~S.}\binits{M.~S.}}
(\byear{1994}).
\btitle{Stable Non-{G}aussian Random Processes: Stochastic Models with
Infinite Variance}.
\bpublisher{Chapman \& Hall}, \blocation{New York}.
\bid{mr={1280932}}
\bptok{imsref}%
\end{bbook}
%
\endbibitem

\end{thebibliography}
\end{document}